%% file: example_paper.tex
\documentclass[singlespace]{article}
\usepackage{inputenc}
\usepackage{hyperref}
\usepackage{url}
\usepackage{amsmath}
\usepackage{mathtools}
\usepackage{bm}
\usepackage{amsfonts}
\usepackage{algorithm}
\usepackage{algorithmic}
\usepackage{amssymb}
\usepackage{amsthm,paralist}
\usepackage{graphicx}
\usepackage{color}
\usepackage{subfigure}   
\usepackage{diagbox} 
\usepackage{stfloats}  
\usepackage{footnote}
 \usepackage{multirow}  
 \usepackage{booktabs} 
\usepackage[scriptsize,bf]{caption}
\usepackage{amssymb}
\usepackage{diagbox} 
\usepackage{slashbox}
\usepackage{stfloats}  
\usepackage{graphics,graphicx}
\usepackage{epstopdf}
\usepackage{cleveref} 
\usepackage{lipsum}
\usepackage{natbib}

\usepackage{amsthm}

\newtheorem{theorem}{Theorem}[section]
\newtheorem{lemma}[theorem]{Lemma}

\newtheorem{prop}{Proposition}
\newtheorem{definition}{Definition}
\newtheorem{assumption}{Assumption}
\newtheorem{remark}{Remark}
\newtheorem{example}{Example}

\usepackage{fancybox}
\usepackage{booktabs}
\usepackage{amssymb}
\usepackage{multirow}
\usepackage{fancyhdr}
\usepackage{xcolor}
\usepackage{bm}
\usepackage{indentfirst}
\usepackage{booktabs}
\usepackage{bigstrut}
\usepackage{subfigure}
\usepackage{float}

\usepackage{tabularx}

\usepackage{microtype}
\usepackage{graphicx}
\usepackage{subfigure}
\usepackage{booktabs}

\usepackage{array}
\newcolumntype{P}[1]{>{\centering\arraybackslash}p{#1}}

\usepackage{array} 
\newcolumntype{C}{>{$\displaystyle}c<{$}}

\DeclareMathOperator*{\minimize}{minimize}

\title{Global Convergence and Variance-Reduced Optimization \\
           for a Class of Nonconvex-Nonconcave Minimax Problems}

\author{
Junchi Yang
\thanks{Department of Industrial and Enterprise Systems Engineering, University of Illinois at Urbana-Champaign, IL 61801, USA 
  (\texttt{junchiy2@illinois.edu, niaohe@illinois.edu}).}
\and 
Negar Kiyavash
\thanks{School of Management of Technology, \'Ecole Polytechnique F\'ed\'erale de Lausanne, Switzerland
  (\texttt{negar.kiyavash@epfl.ch}).}
\and 
Niao He
\footnotemark[1]
}

\ifpdf
\hypersetup{
  pdftitle={blabla},
  pdfauthor={blabla}
}
\fi

\begin{document}
\maketitle

\begin{abstract}
Nonconvex minimax problems appear frequently in emerging machine learning applications, such as generative adversarial networks and adversarial learning. Simple algorithms such as the gradient descent ascent (GDA) are the common practice for solving these nonconvex games and receive lots of empirical success. Yet, it is known that these vanilla GDA algorithms with constant step size can potentially diverge even in the convex setting. In this work, we show that for a subclass of nonconvex-nonconcave objectives satisfying a so-called two-sided Polyak-{\L}ojasiewicz inequality, the alternating gradient descent ascent (AGDA) algorithm converges globally at a linear rate and the stochastic AGDA achieves a sublinear rate. We further develop a variance reduced algorithm that attains a provably faster rate than AGDA when the problem has the finite-sum structure. 
\end{abstract}

\input{Introduction}

\input{Two-sided-PL}

\input{AGDA}

\input{SVRG}

\input{Experiments}

\section{Conclusion}
\noindent In this paper, we identify a subclass of nonconvex-nonconcave minimax problems, represented by the the so-called two-side PL condition, for which AGDA and Stoc-AGDA can converge to \emph{global} saddle points. 
We also propose the first linearly-convergent variance-reduced AGDA algorithm that is provably always faster than AGDA, for this subclass of minimax problems
. 
We hope this work can shed some light on the understanding of nonconvex-nonconcave minimax optimization: (1) different learning rates for two players are essential in GDA algorithms with alternating updates; (2) 
convexity-concavity is not a watershed to guarantee global convergence of GDA algorithms; (3) the complexity of solving minimax problems under PL conditions may have high-order dependence on the condition number in contrast to problems with strong convex-concavity conditions. It remains interesting to explore whether similar results apply to GDA algorithms with simultaneous updates and whether these algorithms can be further accelerated with momentum or catalyst schemes.

\bibliography{example_paper}
\bibliographystyle{plainnat}

\newpage
\onecolumn
\appendix

\input{appendix1.tex}

\input{appendix2.tex}

\input{appendix3.tex}

\input{appendix4.tex}


\end{document}

%% file: Introduction.tex
\section{Introduction}

\noindent We consider minimax optimization problems of the forms
\begin{equation} \label{objective}
    \min_{x\in\mathbb{R}^{d_1}}\max_{y\in\mathbb{R}^{d_2}} f(x, y) \triangleq  \mathbb{E}[F(x, y ; \xi)],
\end{equation}
and 
\begin{equation} \label{objective finite sum}
    \min_{x\in\mathbb{R}^{d_1}}\max_{y\in\mathbb{R}^{d_2}} f(x, y) \triangleq  \frac{1}{n}\sum_{i=1}^n f_i(x, y),
\end{equation}
where $\xi$ is a random vector with support $\Xi$, and $f(x,y)$ is a possibly nonconvex-nonconcave function. Minimax problems have been widely studied in game theory and operations research. Recent emerging applications in machine learning have further stimulated a surge of interest in these problems. For example, generative adversarial networks (GANs) \citep{goodfellow2016deep} can be viewed as a two-player game between a generator that produces synthetic data and a discriminator that differentiates between true data and synthetic data. In reinforcement learning, solving Bellman equations can also be reformulated as minimax optimization problems \citep{chen2016stochastic,dai2017learning,dai2017sbeed}. Other applications include robust optimization \citep{namkoong2016stochastic, namkoong2017variance}, adversarial machine learning \citep{sinha2017certifiable, madry2017towards}, unsupervised learning \citep{xu2005maximum}, and so on. 

The most natural and frequently used methods for solving minimax problems (\ref{objective}) and (\ref{objective finite sum}) are the gradient descent ascent (GDA) algorithms (or their stochastic variants), with either simultaneous or alternating updates of the primal-dual variables, referred to as SGDA and AGDA, respectively, throughout the paper. While these algorithms have received much empirical success especially in adversarial training, it is known that these GDA algorithms with constant stepsizes could fail to converge for general smooth function \citep{mescheder2018training}, even for the bilinear games \citep{gidel2019negative}; even when they do converge, the stable limit point may not be a  local Nash equilibrium \citep{daskalakis2018training, mazumdar2018convergence}. 
On the other hand, GDA algorithms can converge linearly to the saddle point for strongly-convex-strongly-concave functions \citep{facchinei2007finite}. Moreover, for many  simple nonconvex-nonconcave objective functions, such as, 
$f(x,y) = x^2 + 3\sin^2x\sin^2y-4y^2-10\sin^2y,$
we also observe that GDA algorithms with constant stepsizes indeed converge to the global Nash equilibrium (or saddle point), at a linear rate (see Figure 1). This also holds true for their stochastic variants, albeit at a sublinear rate. 
These facts naturally raise a question: 
\emph{Is there a general condition under which  GDA algorithms converge to the global optima?}

\begin{figure} [!ht]  
\centering
\subfigure[Surface plot of $f(x,y)$]{%
\label{fig:first}%
\includegraphics[height=0.215\linewidth]{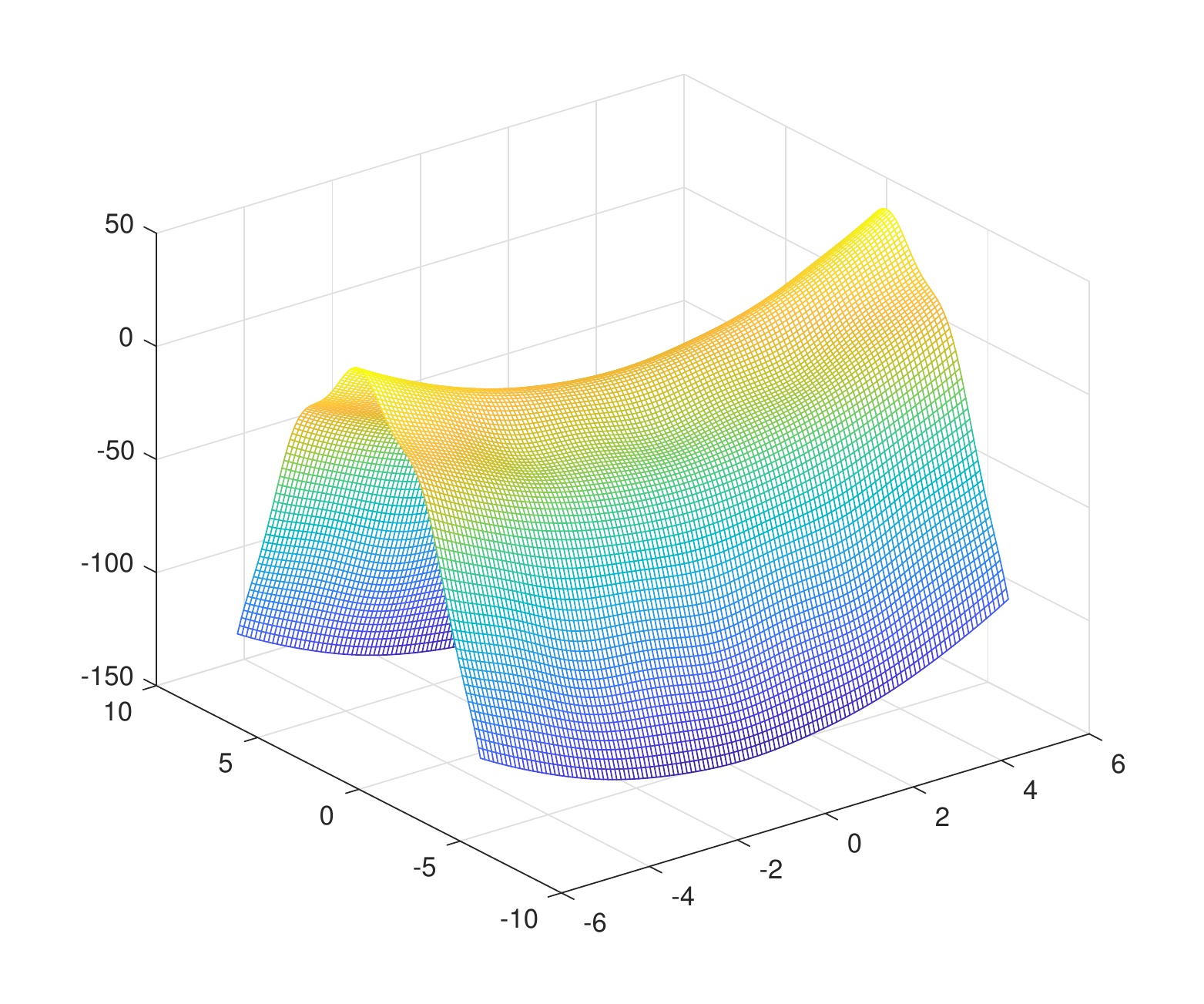}}%
\subfigure[Deterministic GDA algorithms]{%
\label{fig:second}%
\includegraphics[height=0.215\linewidth]{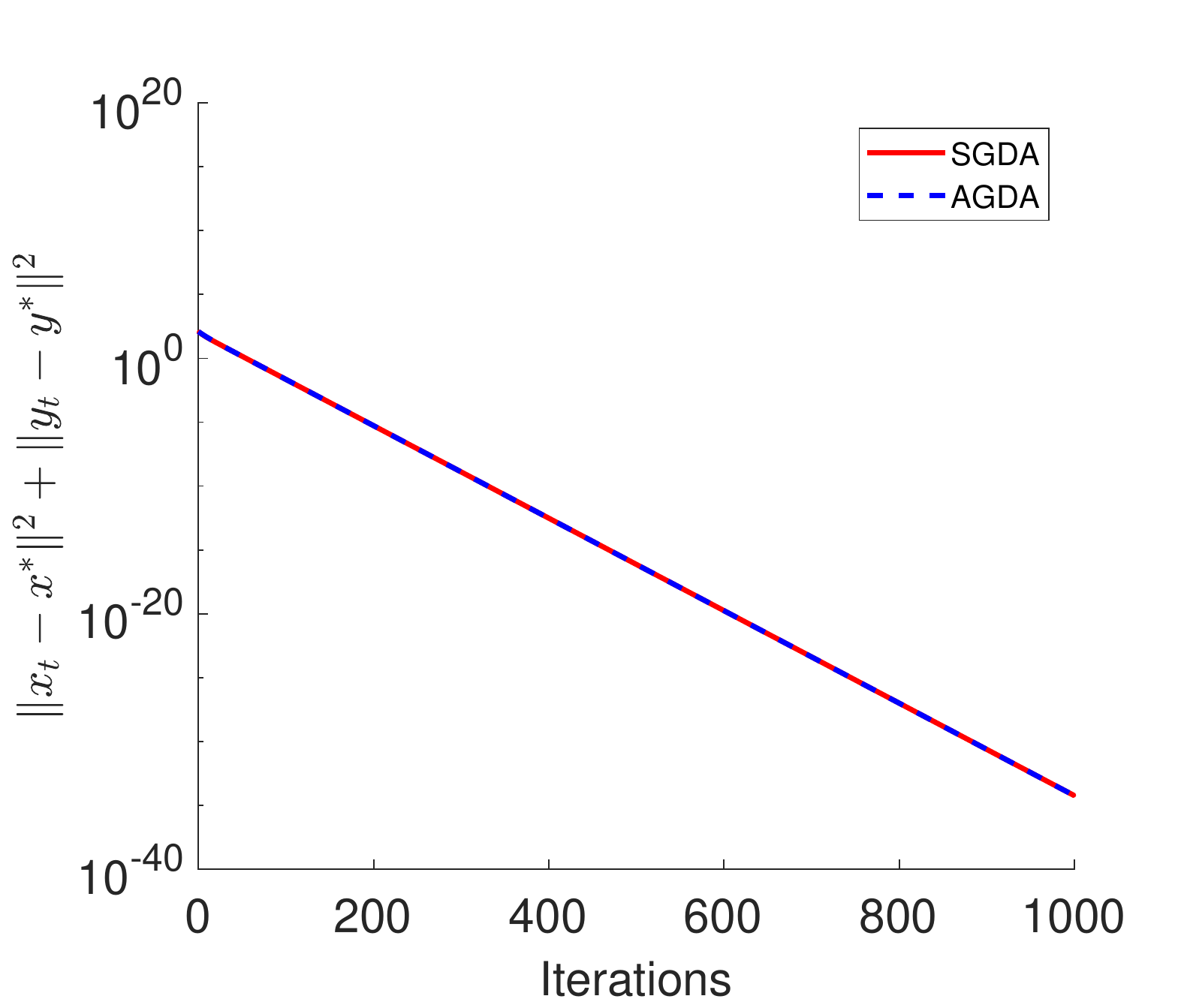}}%
\subfigure[Stochastic GDA algorithms]{%
\label{fig:third}%
\includegraphics[height=0.215\linewidth]{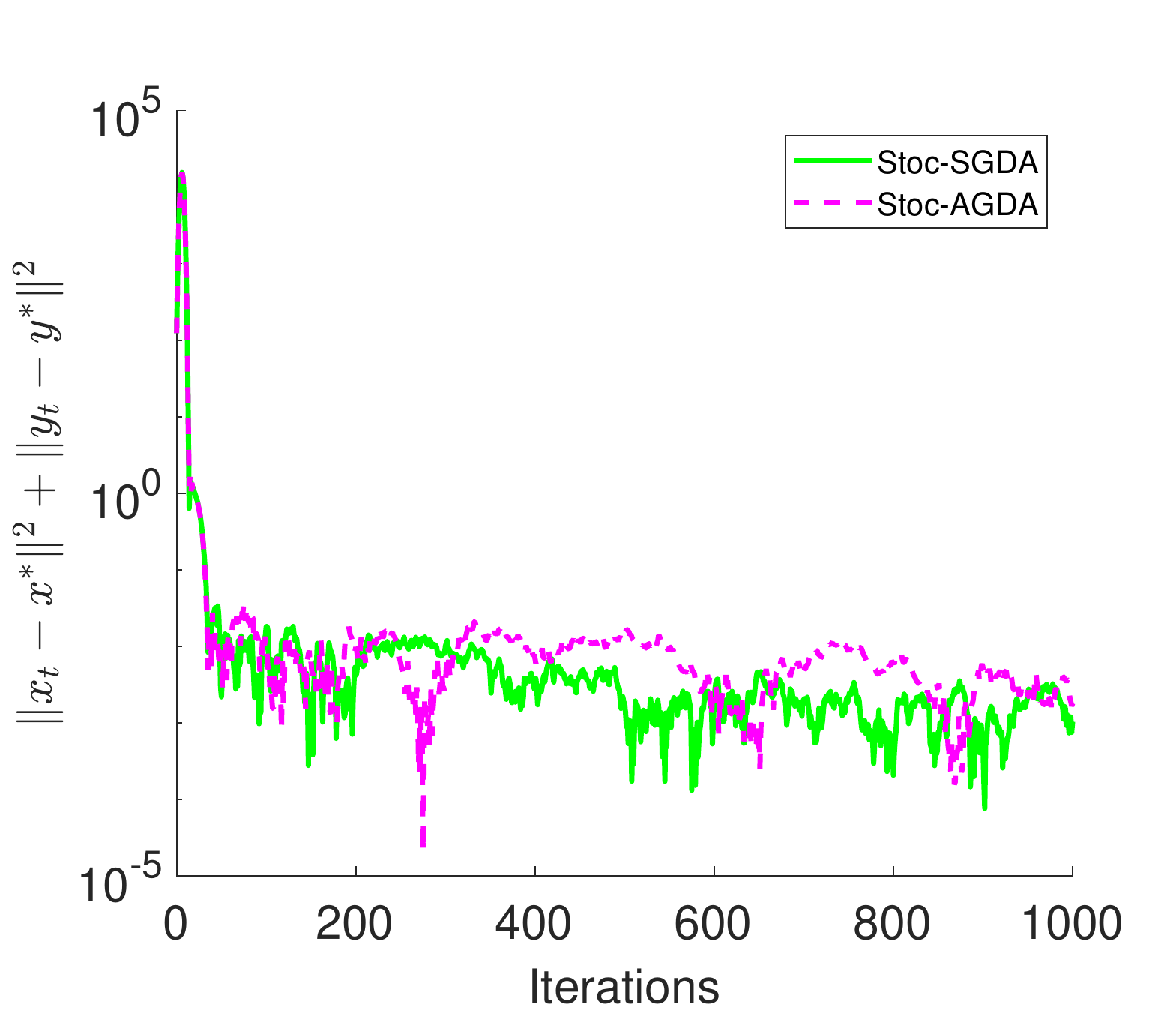}}%
\subfigure[Trajectories]{%
\label{fig:fourth}%
\includegraphics[height=0.215\linewidth]{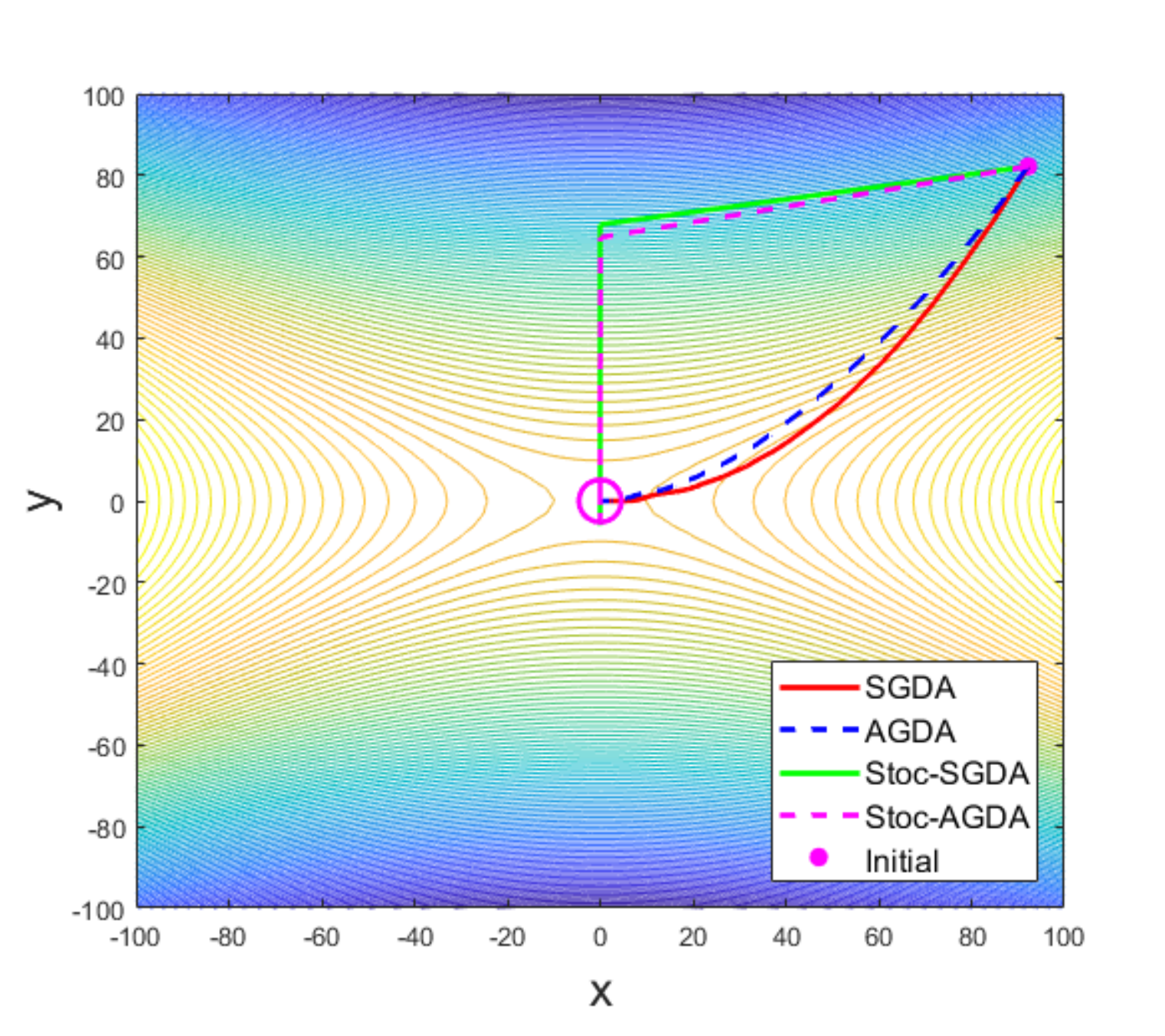}}
\vspace{-0.4cm}
\caption{(a) Surface plot of the nonconvex-nonconcave function $f(x,y) = x^2 + 3\sin^2x\sin^2y-4y^2-10\sin^2y$ ; (b) Convergence of SGDA and AGDA; (c) Convergence of stochastic SGDA and stochastic AGDA; (d) Trajectories of four algorithms} \label{figure 1}
\end{figure}

Furthermore, the use of variance reduction techniques has played a prominent role in improving the convergence over stochastic or batch algorithms for both convex and nonconvex minimization problems, which have been extensively studied in the past few years; see, e.g., \citep{johnson2013accelerating, reddi2016stochastic,reddi2016fast, xiao2014proximal}, just to name a few. However, when it comes to the minimax problems, there are limited results, except under convex-concave setting \citep{palaniappan2016stochastic,du2019linear}. This leads to another open question:
\emph{Can we improve GDA algorithms for nonconvex-nonconcave minimax problems?}

\subsection{Our contributions}
\noindent In this paper, we address these two questions and specifically focus on the alternating gradient descent ascent, namely AGDA. This is due to several considerations.  First of all, it has been recently shown that alternating updates of GDA are more stable than
simultaneous updates~\citep{gidel2019negative,bailey2019finite}. Note that for a convex-concave matrix game, SGDA may diverge while AGDA is proven to always have bounded iterates \citep{gidel2019negative}. See Figure~\ref{figure AvsS} for a simple illustration. Secondly, in general, it is much more challenging to analyze AGDA than SGDA. There is a lack of discussion on the convergence of AGDA for general minimax problems in the literature. Our contributions are summarized as follows. 

\begin{figure*} [!ht]  
\centering
\subfigure[ $\tau = 0.01$]{%
\label{fig:first}%
\includegraphics[height=0.21\linewidth]{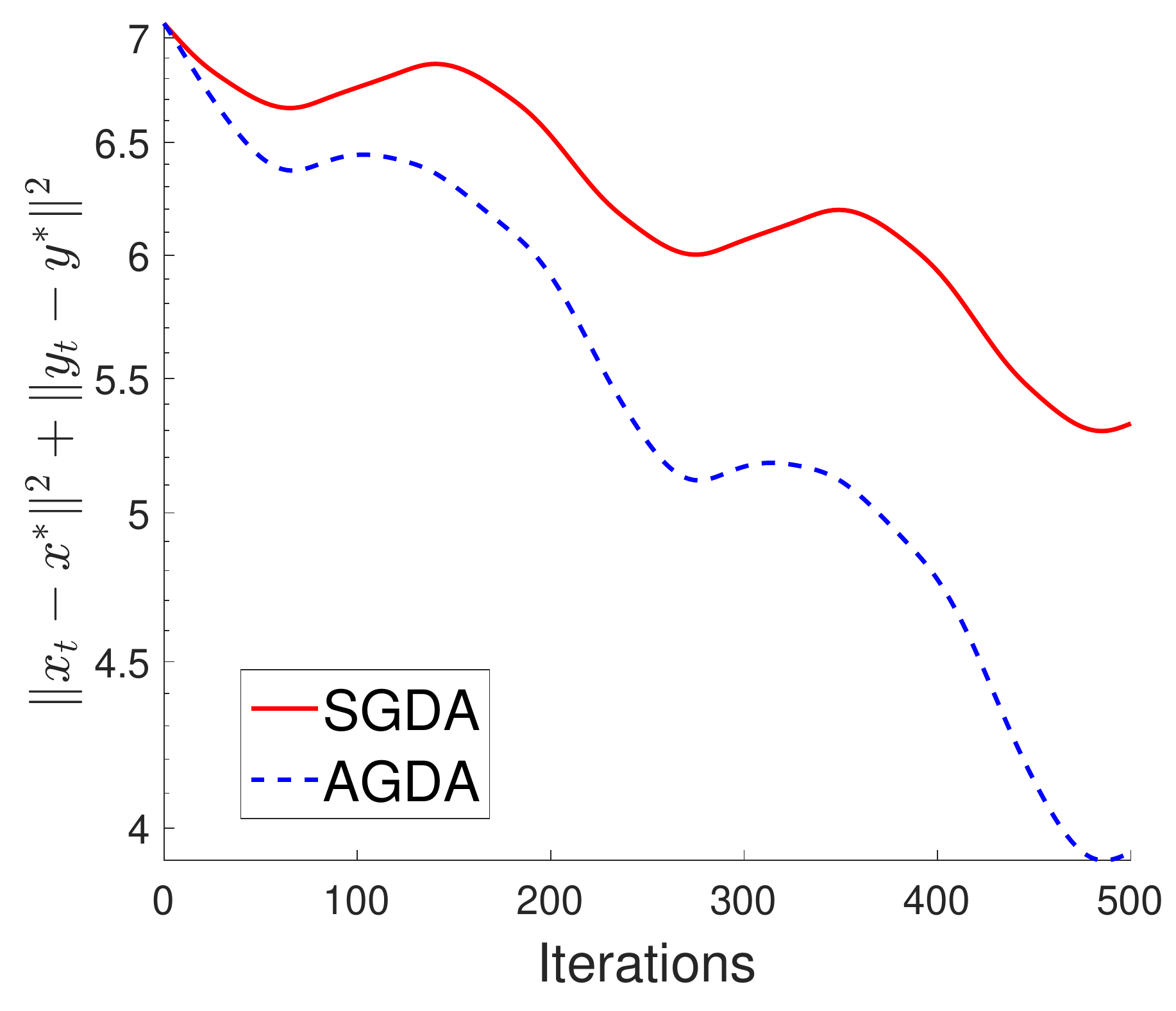}}%
\subfigure[ $\tau = 0.01$]{%
\label{fig:second}%
\includegraphics[height=0.21\linewidth]{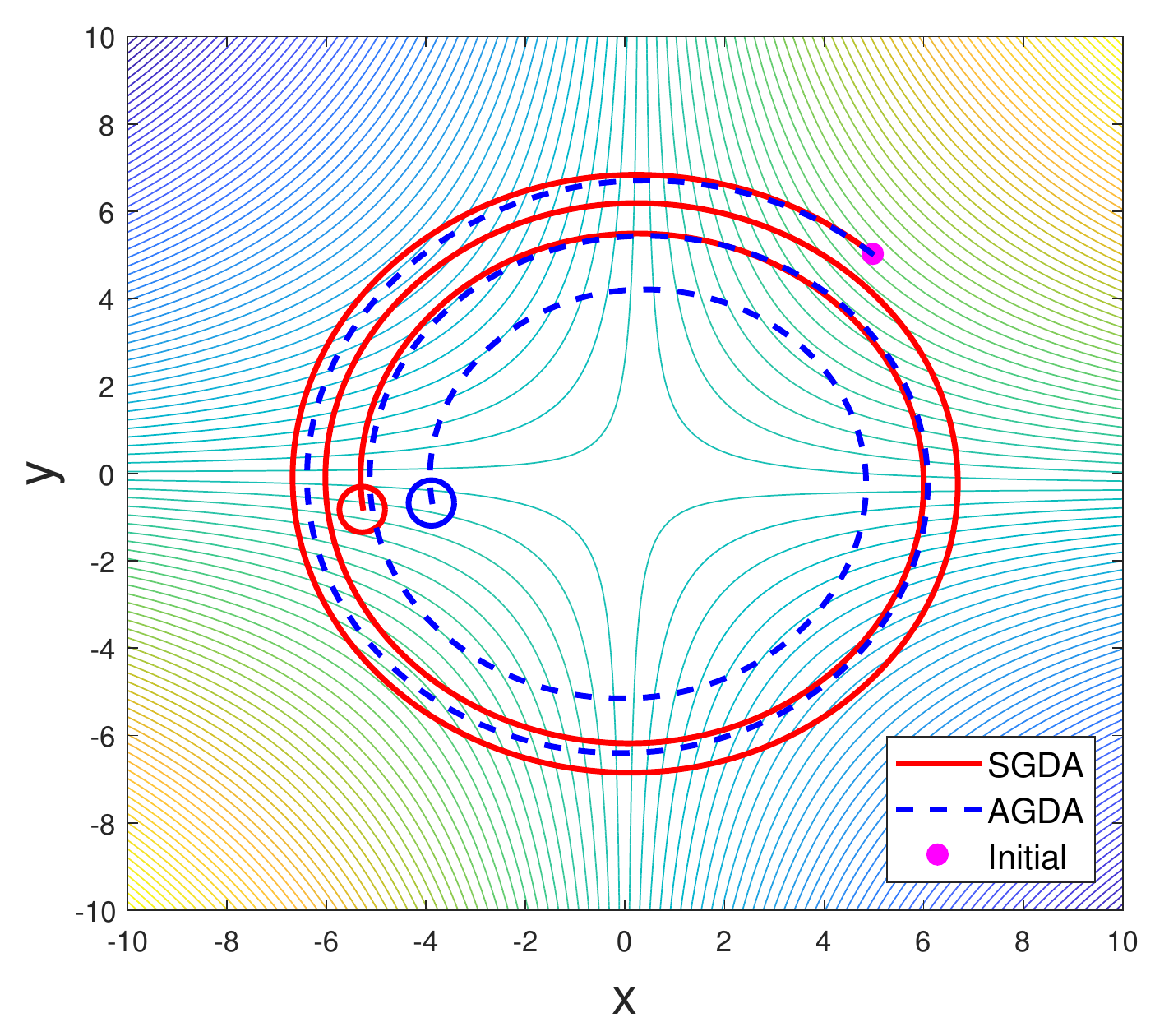}}%
\subfigure[$\tau = 0.025$]{%
\label{fig:third}%
\includegraphics[height=0.21\linewidth]{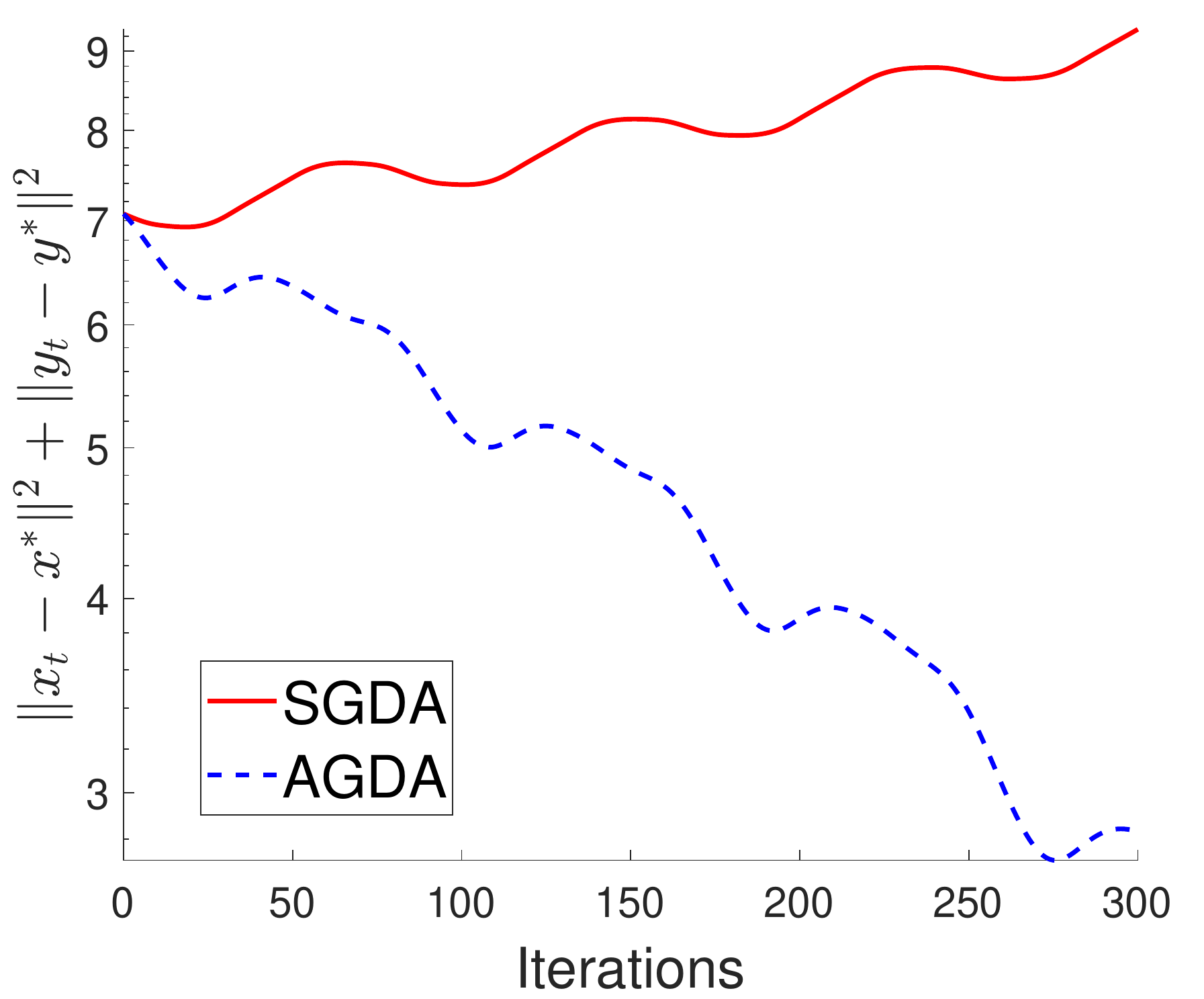}}%
\subfigure[$\tau = 0.025$]{%
\label{fig:fourth}%
\includegraphics[height=0.21\linewidth]{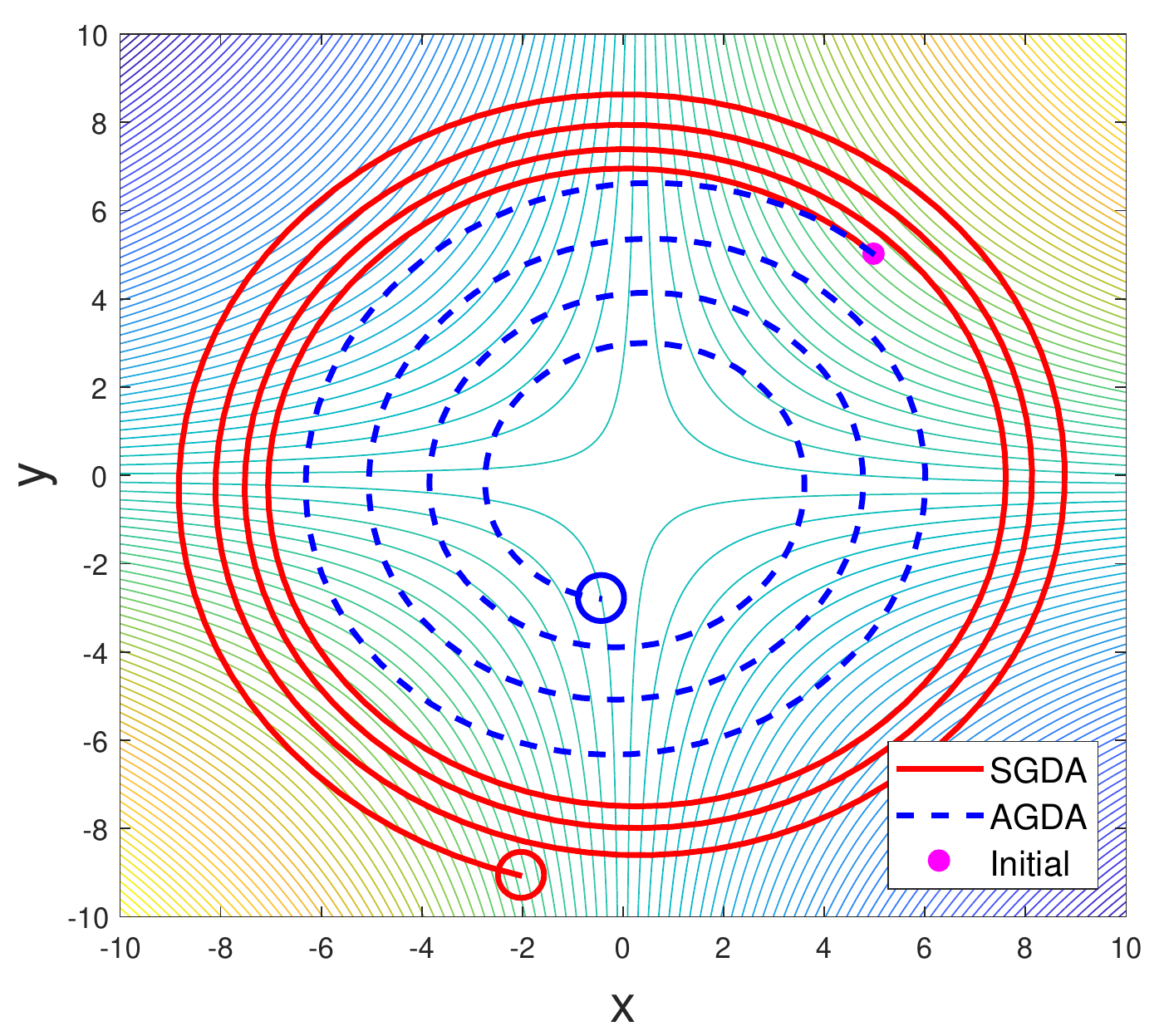}}
\vspace{-0.4cm}
\caption{Consider the objective $f(x, y) = \log\left(1+e^x\right) + 3xy - \log\left(1+e^y\right)$: (a) Convergence of AGDA and SGDA under the stepsize $\tau=0.01$; (b) Trajectories of two algorithms under the stepsize $\tau=0.01$; (c) Convergence of AGDA and SGDA under stepsize $\tau=0.025$; (d) Trajectories of two algorithms with stepsize $\tau=0.025$;} \label{figure AvsS}
\end{figure*}

First, we identity a general condition that relaxes the convex-concavity requirement of the objective function while still guaranteeing global convergence of AGDA and stochastic AGDA (Stoc-AGDA).  We call this the two-sided PL condition, which requires that both players' utility functions satisfy Polyak-{\L}ojasiewicz (PL) inequality \citep{polyak1963gradient}.
Such conditions indeed hold true for several applications, including robust least square, generative adversarial imitation learning for linear quadratic regulator (LQR) dynamics~\citep{cai2019global}, and potentially many others in adversarial learning~\citep{du2019gradient}, robust phase retrieval~\citep{sun2018geometric,zhou2016geometrical}, robust control~\citep{fazel2018global}, and etc.
We show that under the two-sided PL condition, AGDA with proper constant stepsizes converges globally to a saddle point at a linear rate of $\mathcal{O}(1-\kappa^{-3})^t$, while  Stoc-AGDA with proper diminishing stepsizes  converges to a saddle point at a sublinear rate of $\mathcal{O}(\kappa^5/t)$, where $\kappa$ is the underlying condition number.

Second, for minimax problems with the finite sum structure, we introduce a variance-reduced AGDA algorithm (VR-AGDA) that leverages the idea of stochastic variance reduced gradient (SVRG)~\citep{johnson2013accelerating,reddi2016stochastic} with the alternating updates. We prove that VR-AGDA achieves the complexity of $\mathcal{O}((n^{2/3}\kappa^3 \log(1/\epsilon))$ in the region $n\leq \kappa^9$ and $\mathcal{O}(n+\kappa^9)\log(1/\epsilon))$ in the region $n\geq \kappa^9$, where $n$ is the number of component functions. This greatly improves over the $\mathcal{O}\left(n\kappa^3\log\frac{1}{\epsilon}\right)$ complexity of AGDA when applied to the finite sum minimax problems. We summarize the results of these algorithms in Table \ref{table:1}. Our numerical experiments further demonstrate that VR-AGDA performs significantly better than AGDA and Stoc-AGDA, especially for problems with large condition numbers.  To our best knowledge, this is the first work to provide a variance reduced algorithm and theoretical guarantees in the nonconvex-nonconcave regime of minimax optimization. 


\begin{table}[!ht] 
\center
\centering
\vskip 0.15in
 \begin{tabular}{ |P{6em} | P{8em} | P{8em}| P{12em} | }  
\hline
 Algorithms & AGDA  & Stoc-AGDA &   VR-AGDA \\ 
\hline
Complexity & $\mathcal{O}\left(n\kappa^3\log\frac{1}{\epsilon}\right)$ & $\mathcal{O}\left(\frac{\kappa^5}{\mu_2\epsilon}\right)$  & $
                \begin{array}{ll}
                  \mathcal{O}\left(n^{\frac{2}{3}}\kappa^3\log\frac{1}{\epsilon}\right),  n\leq \kappa^9\\
                  \mathcal{O}\left((n+\kappa^9)\log\frac{1}{\epsilon}\right)
                \end{array}
             $ \\ 
\hline
\end{tabular}
\caption{Complexities of three algorithms for the finite-sum problem~(\ref{objective finite sum}), where $\kappa\triangleq l/\mu_1$ is condition number, $l$ is Lipschtiz gradient constant, $\mu_1$ and $\mu_2$ are the two-side PL constants with $\mu_1\leq \mu_2$. See Section \ref{Sec4} for more details.}  \label{table:1}
\end{table}

\subsection{Related work}

\paragraph{Nonconvex minimax problems.} There has been a recent surge in research on solving minimax optimization beyond the convex-concave regime~\citep{sinha2017certifiable, chen2017robust, qian2019robust, thekumparampil2019efficient,  lin2018solving, nouiehed2019solving, abernethy2019last}, but they differ from our work from various perspectives. For example, \citet{chen2017robust,sinha2017certifiable,lin2019gradient, thekumparampil2019efficient} considered the minimax problem when the objective function is nonconvex in $x$ but concave in $y$ and focused on achieving convergence to stationary points. Their algorithms require solving the inner maximization or some sub-problems with high accuracy at every iteration, which are different from AGDA. \citet{lin2018solving} considered a general class of weakly-convex weakly-concave minimax problems and proposed an inexact proximal point method to find an $\epsilon$-stationary point. Their convergence result relies on assuming the existence of a solution to the corresponding Minty variational inequality, which is often hard to verify. \citet{abernethy2019last} recently showed the linear convergence of a second-order iterative algorithm, called Hamiltonian gradient descent (HGD), for a subclass of ``sufficiently bilinear" functions. Compared with their work, the  PL condition we consider in this paper is easier to verify and GDA algorithms are much simpler.
\vspace{-0.25cm}
\paragraph{PL condition.} Recently, \citet{nouiehed2019solving} studied a class of minimax problems where the objective only satisfies a one-sided PL condition and introduced the GDmax algorithm, which takes multiple ascent steps at every iteration. Our work differs from \citep{nouiehed2019solving} in two aspects: (i) we consider the two-sided PL condition which guarantees global convergence \footnote{We also show that AGDA can find $\epsilon-$stationary point for minimax problems under the one-sided PL condition within $\mathcal{O}(1/\epsilon^2)$ iterations in Appendix \ref{appendix4}.}; (ii) we consider AGDA which takes one ascent step at every iteration. Another closely related work is \citet{cai2019global}. The authors considered a specific application in generative adversarial imitation learning with linear quadratic regulator dynamics. This is a special example that falls under the two-sided PL condition.

\vspace{-0.25cm}
\paragraph{Variance-reduced minimax optimization.} There exists a few works that apply variance reduction techniques to minimax optimization. \citet{palaniappan2016stochastic, luo2019stochastic} provided linear-convergent algorithms for strongly-convex-strongly-concave objectives, based on simultaneous updates. \citet{du2019linear} extended the result to convex-strongly-concave objectives with full-rank coupling bilinear term. In contrast, we are dealing with a much broader class of objectives that are possibly nonconvex-nonconcave. We point out that \citet{luo2020stochastic} recently introduced a variance-reduced algorithm for finding the stationary point of nonconvex-strongly-concave problems, which is again different from our setting.


The rest of this paper is organized as follows. In Section \ref{Sec2}, we introduce the two-sided PL condition and show the equivalence of three min-max optimality criteria under this condition. In Section \ref{Sec3}, we describe deterministic and stochastic AGDA algorithms, and provide convergence analyses of those algorithms under the two-sided PL condition. In Section \ref{Sec4}, we introduce the variance-reduced AGDA algorithm and establish its convergence results.  In Section \ref{Sec5}, we provide numerical performance of these algorithms for robust least square and imitation learning for LQR.

%% file: Two-sided-PL.tex
\section{Global optima and two-sided PL condition} \label{Sec2}

\noindent Throughout this paper, we assume that the function $f(x,y)$ in  (\ref{objective}) is continuously differentiable and has Lipschitz gradient. We state it as a basic assumption.  Here $\|\cdot\|$ is used to denote the Euclidean norm. 
\begin{assumption} [Lipschitz gradient]
There exists a positive constant $l>0$ such that
\begin{align*}
     &\left\| \nabla _ { x } f \left( x _ { 1 } , y _ {1} \right) - \nabla _ { x } f \left( x _ { 2 } , y _ {2} \right) \right\| \leq l [\left\| x _ { 1 } - x _ { 2 } \right\| + \left\| y _ { 1 } - y _ { 2 } \right\|],\\
    &\left\| \nabla _ { y } f \left( x _{1} , y _ { 1 } \right) - \nabla _ { y } f \left( x_ {2} , y _ { 2 } \right) \right\| \leq l [\left\| x _ { 1 } - x _ { 2 } \right\| + \left\| y _ { 1 } - y _ { 2 } \right\|],
\end{align*} \label{Lipscthitz gradient}
holds for all $x_1, x_2\in \mathbb{R}^{d_1}, y_1$, $y_2 \in \mathbb{R}^{d_2}$.
\end{assumption}

We now define three notions of optimality for minimax problems. The most direct notion of optimality is global minimax point, at which $x^*$ is an optimal solution to the function $g(x): = \max_{y}f(x, y)$ and $y^*$ is an optimal solution to $\max_{y} f(x^*, y)$. In the two-player zero-sum game, the notion of saddle point is also widely used \citep{von2007theory, nash1953two}. For a saddle point $(x^*, y^*)$, $x^*$ is an optimal solution to $\min_{x} f(x, y^*)$ and $y^*$ is an optimal solution to $\max_{y} f(x^*, y)$.

\begin{definition}[Global optima]\quad\begin{enumerate}
\itemsep-0.9em 
    \item $(x^*, y^*)$ is a global minimax point, if for any $(x,y):$
     \vspace{-1 mm}
 \begin{equation}
     f(x^*,y) \leq f(x^*, y^*) \leq \max_{y'} f(x,y').
 \end{equation}
 \item  $(x^*, y^*)$ is a saddle point, if for any $(x,y):$
 \begin{equation}
    f(x^*,y) \leq f(x^*, y^*) \leq  f(x,y^*).
 \end{equation} 
 \item $(x^*, y^*)$ is a stationary point, if $:$
 \begin{equation}
     \nabla_xf(x^*,y^*) = \nabla_yf(x^*,y^*) = 0.
 \end{equation}
\end{enumerate}
\end{definition}

For general nonconvex-nonconcave minimax problems, these three notions of optimality are not necessarily equivalent. A stationary point may not be a saddle point or a global minimax point; a global minimax point may not be a saddle point or a stationary point. Note that generally speaking, for minimax problems, a saddle point or a global minimax point may not always exist. However, since our goal in this paper is to find global optima, in the remainder of the paper, we assume that a saddle point always exists. 

\begin{assumption} [Existence of saddle point]
The objective function $f$ has at least one saddle point. We also assume that for any fixed $y$, $\min_{x\in\mathbb{R}^{d_1}} f(x,y)$ has a nonempty solution set and a  optimal value, and for any fixed $x$, $\max_{y\in\mathbb{R}^{d_2}} f(x,y)$ has a nonempty solution set and a finite optimal value. \label{Existence of global minimax point}
\end{assumption}

For unconstrained minimization problems: $\min_{x\in\mathbb{R}^n} f(x)$, \citet{polyak1963gradient} proposed Polyak-{\L}ojasiewicz (PL) condition, which is sufficient to show global linear convergence for gradient descent without assuming convexity. Specifically, a function $f(\cdot)$ satisfies PL condition if it has a nonempty solution set and a finite optimal value $f^*$, and there exists some $\mu>0$ such that 
$    \frac{1}{2}\Vert \nabla f(x)\Vert^2 \geq \mu(f(x)-f^*), \forall x
$. As discussed in~\citet{karimi2016linear}, PL condition is weaker, or not stronger, than other well-known conditions that guarantee linear convergence for gradient descent, such as error bounds (EB) \citep{luo1993error}, 
weak strong convexity (WSC) \citep{necoara2018linear} and restricted secant inequality (RSI) \citep{zhang2013gradient}.

We introduce a straightforward generalization of the PL condition to the minimax problem: function $f(x, y)$ satisfies the PL condition with constant $\mu_1$ with respect to $x$, and -$f$ satisfies PL condition with constant $\mu_2$ with respect to $y$. We formally state this in the following definition.

\begin{definition} [Two-sided PL condition]
A continuously differentiable function $f(x,y)$ satisfies the two-sided PL condition if there exist constants $\mu_1, \mu_2>0$ such that:
\begin{align*} \label{two-sided PL}
    \Vert \nabla_xf(x,y) \Vert^2 \geq 2\mu_1 [f(x, y)-\min_{x} f(x,y)], \forall x,y,\\
   \Vert \nabla_yf(x,y) \Vert^2 \geq 2\mu_2 [\max_{y}f(x, y)-f(x,y)], \forall x,y.
\end{align*} 
\end{definition}
 
The two-sided PL condition does not imply convexity-concavity, and it is a much weaker condition than strong-convexity-strong-concavity.  In Lemma \ref{equivalent optimality}, we show that three notions of optimality are equivalent under the two-sided PL condition. Note that they may not be unique. 
 
\begin{lemma} 
If the objective function $f(x,y)$ satisfies the two-sided PL condition, then the following holds true:
\begin{equation*}
    \text{(saddle point)} \Leftrightarrow \text{(global minimax)} \Leftrightarrow \text{(stationary point)}.
\end{equation*} \label{equivalent optimality}
\end{lemma}
\vspace{-7 mm}
\noindent Below we give some examples that satisfy this condition.
\begin{example}
The nonconvex-nonconcave function in the introduction,
$
    f(x,y) = x^2 + 3\sin^2x\sin^2y-4y^2-10\sin^2y  
$
satisfies the two-sided PL condition with $\mu_1 = 1/16, \mu_2 = 1/11$ (see Appendix \ref{appendix1}). 
\end{example}

\begin{example} \label{ex 2}
$f(x, y) = F(Ax, By)$, where $F(\cdot, \cdot)$ is strongly-convex-strongly-concave and $A$ and $B$ are arbitrary matrices, satisfies the two-sided PL condition.
\end{example}

\begin{example} 
The generative adversarial imitation learning for LQR can be formulated as $\min_K\min_{\theta} m(K, \theta)$, where $m$ is strongly-concave in terms of $\theta$ and satisfies PL condition in terms of $K$ (see \citep{cai2019global} for more details), thus satisfying the two-sided PL condition. 
\end{example}

Under the two-sided PL condition, the function $g(x) := \max_{y}f(x, y)$  can be shown to satisfy PL condition with $\mu_1$ (see Appendix \ref{appendix1}). Moreover, it holds that $g$ is also $L$-smooth with $L: =l+ l^2/\mu_2$~\citep{nouiehed2019solving}. 
Finally, we denote $\mu=\min(\mu_1,\mu_2)$ and $\kappa=\frac{l}{\mu}$, which represents the condition number of the problem.

%% file: AGDA.tex
\section{Global convergence of AGDA and Stoc-AGDA}
\label{Sec3}

\noindent In this section, we establish the convergence rate of the stochastic alternating gradient descent ascent (Stoc-AGDA) algorithm, which we present in Algorithm \ref{s-agda}, under the two-sided PL condition. Stoc-AGDA updates variables $x$ and $y$ sequentially using stochastic gradient descent/ascent steps. Here we make standard assumptions about stochastic gradients $G_x(x, y, \xi)$ and $G_y(x, y, \xi)$.
\begin{assumption}[Bounded variance]\label{stochastic gradients} 
$G_x(x,y, \xi)$ and $G_y(x,y, \xi)$ are unbiased stochastic estimators of $\nabla_x f(x, y)$ and $\nabla_y f(x, y)$ and have variances bounded by $\sigma^2>0$.  
\end{assumption}

\begin{algorithm}[ht] 
    \caption{ Stoc-AGDA}
    \begin{algorithmic}[1]
        \STATE Input: $(x_0,y_0)$, step sizes $\{\tau_1^t\}_t>0, \{\tau_2^t\}_t>0$
        \FORALL{$t = 0,1,2,...$}
            \STATE Draw two i.i.d. samples $\xi_{t1}, \xi_{t2}\sim P(\xi)$ 
            \STATE $x_{t+1}\gets  x_t-\tau_1^t G_x(x_t,y_t, \xi_{t1})$
            \STATE $y_{t+1}\gets y_t+\tau_2^t G_y(x_{t+1},y_t, \xi_{t2})$
        \ENDFOR
    \end{algorithmic} \label{s-agda}
\end{algorithm}

Note that Stoc-AGDA with constant stepsizes (i.e., $\tau_1^t = \tau_1$ and $\tau_2^t = \tau_2$) and noiseless stochastic gradient (i.e., $\sigma^2=0$) reduces to AGDA: 
\vspace{-2mm}
 \begin{align}
 x_{t+1} &= x_t - \tau_1\nabla_xf(x_t, y_t), \\
 y_{t+1} &= y_t - \tau_2\nabla_y f(x_{t+1}, y_t).
 \end{align}
We will measure the inaccuracy of $(x_t,y_t)$ through the potential function
\begin{equation}
    P_t :=  a_t + \lambda\cdot b_t,
\end{equation}
where
$a_t = \mathbb{E}[g(x_t)-g^*], b_t = \mathbb{E}[g(x_t) - f(x_t, y_t)]$ and $\lambda>0$ to be specified later in the theorems. Recall that $g(x) := \max_y f(x,y)$ and $g^* = \min_xg(x)$. This metric is driven by the definition of minimax point, because $g(x)-g^*$ and $g(x) - f(x, y)$ are non-negative for any $(x,y)$, and both equal to 0 if and only if $(x,y)$ is a minimax point. 

\paragraph{Stoc-AGDA with constant stepsizes} We first consider Stoc-AGDA with constant stepsizes. We show that $\{(x_t, y_t)\}_t$ will converge linearly to a neighbourhood of the optimal set. 

\begin{theorem} \label{main} 
Suppose Assumptions \ref{Lipscthitz gradient}, \ref{Existence of global minimax point}, \ref{stochastic gradients} hold and $f(x, y)$ satisfies the two-sided PL condition with $\mu_1$ and $\mu_2$. Define $P_t : = a_t + \frac{1}{10}b_t$. If we run Algorithm \ref{s-agda} with $\tau_2^t = \tau_2 \leq \frac{1}{l}$ and $ \tau_1^t=\tau_1 \leq \frac{\mu_2^2\tau_2}{18l^2}$, then
\vspace{-2 mm}
\begin{align} 
    P_t \leq& (1-\frac{1}{2}\mu_1\tau_1)^tP_0 + \delta,  \label{s-agda convergence}
\end{align}
where $\delta=\frac{(1-\mu_2\tau_2)(L+l)\tau_1^2+l\tau_2^2+10L\tau_1^2}{10\mu_1\tau_1}\sigma^2.$
\end{theorem}

\begin{remark}
In the theorem above, we choose $\tau_1$ smaller than $\tau_2$, $\tau_1/\tau_2\leq \mu_2^2/(18l^2)$, because our potential function is not symmetric about $x$ and $y$. Another reason is because we want $y_t$ to approach $y^*(x_t) \in \arg\max_yf(x_t,y)$ faster so that $\nabla_xf(x_t, y_t)$ is a better approximation for $\nabla g(x_t)$ ($\nabla g(x) = \nabla_xf(x, y^*(x))$, see \citet{nouiehed2019solving}). Indeed, it is common to use different learning rates  for $x$ and $y$ in GDA algorithms for nonconvex minimax problems; see e.g., \citet{jin2019local} and \citet{lin2019gradient}.
Note that  the ratio between these two learning rates is quite crucial here. We also observe empirically when the same learning rate is used, even if small, the algorithm may not converge to saddle points.
\end{remark}

\begin{remark}
When $t\rightarrow \infty$, $P_t\to \delta$. If $\tau_1\rightarrow 0$ and $\tau_2^2/\tau_1 \rightarrow 0$, the error term $\delta$ will go to 0. When using smaller stepsizes, the algorithm reaches a smaller neighbour of the saddle point yet at the cost of a slower rate, as the contraction factor also deteriorates.
\end{remark}

\paragraph{Linear convergence of AGDA}  Setting $\sigma^2=0$, it follows immediately from the previous theorem that AGDA converges linearly under the two-sided PL condition. Moreover, we have 

\begin{theorem} \label{main deterministic}
Suppose Assumptions \ref{Lipscthitz gradient}, \ref{Existence of global minimax point} hold and $f(x, y)$ satisfies the two-sided PL condition with $\mu_1$ and $\mu_2$. Define $P_t : = a_t + \frac{1}{10}b_t$. If we run AGDA with $\tau_1 = \frac{\mu_2^2}{18l^3}$ and $\tau_2 = \frac{1}{l}$,
then 
\begin{equation}
    P_t \leq \left(1-\frac{\mu_1\mu_2^2}{36l^3}\right)^tP_0.
\end{equation}
Furthermore, $\{(x_t, y_t)\}_t$ converges to some saddle point $(x^*, y^*)$, and 
\begin{equation}
    \left\|x_{t}-x^{*}\right\|^{2}+\left\|y_{t}-y^{*}\right\|^{2} \leq \alpha \left(1-\frac{\mu_1\mu_2^2}{36l^3}\right)^tP_0,
\end{equation}
where $\alpha$ is a constant depending on $\mu_1, \mu_2$ and $l$. 
\end{theorem}

 The above theorem implies that the limit point of $\{(x_t, y_t)\}_t$ is a saddle point and the distance to the saddle point decreases in the order of $\mathcal{O}\left((1-\kappa^{-3})^t\right)$. Note that in the special case when the objective is strongly-convex-strongly-concave, it is known that SGDA (GDA with simultaneous updates) achieves an $\mathcal{O}(\kappa^2\log(1/\epsilon))$ iteration complexity (see, e.g., ~\citet{facchinei2007finite}) and this can be further improved to $\mathcal{O}(\kappa\log(1/\epsilon))$ by extragradient methods~\citep{korpelevich1976extragradient}, Nesterov's dual extrapolation~\citep{nesterov2006solving} or accelerated proximal point algorithm \citep{lin2020near}.
However, these result relies heavily on the strong monotonicity of the corresponding variational inequality. For  the general two-sided PL condition,  we may not achieve the same dependency on $\kappa$. 

 

\paragraph{Stoc-AGDA with diminishing stepsizes} While Stoc-AGDA with constant stepsizes only converges linearly to a neighbourhood of the saddle point, Stoc-AGDA with diminishing stepsizes converges to the saddle point but at a sublinear rate $\mathcal{O}(1/t)$. 

\begin{theorem} \label{main diminishing}
Suppose Assumptions \ref{Lipscthitz gradient}, \ref{Existence of global minimax point}, \ref{stochastic gradients} hold and $f(x, y)$ satisfies the two-sided PL condition with $\mu_1$ and $\mu_2$.  Define $P_t  = a_t + \frac{1}{10}b_t$. If we run algorithm \ref{s-agda} with stepsizes $\tau_1^t = \frac{\beta}{\gamma+t}$ and $\tau_2^t = \frac{18l^2\beta}{\mu_2^2(\gamma+t)}
$
for some $\beta > 2/\mu_1$ and $\gamma>0$ such that $\tau_1^1 \leq \min\{1/L, \mu_2^2/18l^2\}$, then we have 
\begin{equation} \label{stoc-agda rate diminishing}
    P_t \leq \frac{\nu}{\gamma+t},
\end{equation}
where $\nu :=$
\begin{equation*}
     \max\Big\{\gamma P_0, \frac{\big[  (L+l)\beta^2+18^2l^5\beta^2/\mu_2^4+10L\beta^2 \big]\sigma^2}{10\mu_1\beta - 20}  \Big\}.
\end{equation*}
\end{theorem}

\begin{remark} \label{remark diminishing}
Note the rate is affected by $\nu$, and the first term in the definition of $\nu$ is controlled by the initial point. In practice, we can find a good initial point by running Stoc-AGDA with constant stepsizes so that only the second term in the definition of $\nu$ matters. Then by choosing $\beta = 3/\mu_1$, we have $\nu = \mathcal{O}\left(\frac{l^5\sigma^2}{\mu_1^2\mu_2^4} \right)$. Thus, the convergence rate of Stoc-AGDA is  $\mathcal{O}\left(\frac{\kappa^5\sigma^2}{\mu  t}\right)$.
\end{remark}

%% file: SVRG.tex
\section{Stochastic variance reduced algorithm}
\label{Sec4}

\noindent In this section, we study the minimax problem in (\ref{objective finite sum}) with the finite-sum structure:
\begin{eqnarray*} 
    \min_x \max_y f(x,y) = \frac{1}{n}\sum_{i=1}^n f_i(x,y),
\end{eqnarray*}
which arises ubiquitously in machine learning. We are especially interested in the case when $n$ is large. We assume the overall objective function $f(x,y)$ still satisfies the two-sided PL condition with $\mu_1$ and $\mu_2$, but we do not assume each $f_i$ to satisfy the two-sided PL condition.  Instead of Assumption~\ref{Lipscthitz gradient}, we now assume each component $f_i$ has Lipschitz gradients. 
\begin{assumption} \label{Lipscthitz gradient 2}
Each $f_i$ has l-Lipschitz gradients. 
\end{assumption}

If we run AGDA with full gradients to solve the finite-sum minimax problem, the total complexity for finding an $\epsilon$-optimal solution is $\mathcal{O}(n\kappa^3\log(1/\epsilon))$ by Theorem \ref{main deterministic}. Despite the linear convergence, the per-iteration cost is high and the complexity can be huge when the number of components $n$ and condition number $\kappa$ are large. 
Instead, if we run Stoc-AGDA, this leads to 
the total complexity $\mathcal{O}\left(\frac{\kappa^5\sigma^2}{\mu_2\epsilon} \right)$ by Remark \ref{remark diminishing}, which has worse dependence on $\epsilon$.

\begin{algorithm}[t]
    \caption{ VR-AGDA}
    \begin{algorithmic}[1]
        \STATE input: $(\Tilde{x}_0, \Tilde{y}_0)$, stepsizes $\tau_1, \tau_2$,  iteration numbers $N, T$
        \FORALL{$k = 0,1,2,...$}
        \FORALL{$t = 0,1,2,...T-1$}
            \STATE $x_{t,0} = \Tilde{x}_t,\quad y_{t,0} = \Tilde{y}_t$,
            \STATE compute $\nabla_x f(\Tilde{x}_t, \Tilde{y}_t) = \frac{1}{n}\sum_{i=1}^n \nabla_x f_i(\Tilde{x}_t, \Tilde{y}_t)$
            \STATE compute $\nabla_y f(\Tilde{x}_t, \Tilde{y}_t) = \frac{1}{n}\sum_{i=1}^n \nabla_y f_i(\Tilde{x}_t, \Tilde{y}_t)$
            \FORALL{$j = 0$ to $N-1$}
                \STATE sample  i.i.d. indices $i_j^1, i_j^2$ uniformly from $[n]$
                \STATE  $ x_{t, j+1} = x_{t,j}-\tau_1 [\nabla_x f_{i_j^1}(x_{t,j},y_{t,j}) -  \nabla_x f_{i_j^1}(\Tilde{x}_t, \Tilde{y}_t) +\nabla_x f(\Tilde{x}_t, \Tilde{y}_t)] $ 
                 \STATE $y_{t, j+1} = y_{t,j}+\tau_2 [\nabla_y f_{i_j^2}(x_{t,j+1}, y_{t,j}) -\nabla_y f_{i_j^2}(\Tilde{x}_t, \Tilde{y}_t) +\nabla_y f(\Tilde{x}_t, \Tilde{y}_t)]$ \label{step10}
            \ENDFOR
            \STATE $\Tilde{x}_{t+1} = x_{t, N},\quad \Tilde{y}_{t+1} = y_{t, N}$  
        \ENDFOR
        \STATE choose $(x^k, y^k)$ from  $\{\{(x_{t,j}, y_{t,j})\}_{j=0}^{N-1}\}_{t = 0}^{T-1}$ uniformly at random 
        \STATE $\Tilde{x}_0 = x^k,\quad \Tilde{y}_0 = y^k$
        \ENDFOR
    \end{algorithmic} \label{svrg}
\end{algorithm}

Motivated by the recent success of stochastic variance reduced gradient (SVRG) technique~\citep{johnson2013accelerating,reddi2016stochastic,palaniappan2016stochastic}, we introduce the VR-AGDA algorithm (presented in Algorithm \ref{svrg}), that combines AGDA with SVRG so that the linear convergence is preserved while improving the dependency on $n$ and $\kappa$. 
VR-AGDA can be viewed as the applying SVRG to AGDA with restarting: at every epoch $k$, we restart the SVRG subroutine (with $T$ outer iterations, $N$ inner steps) by initializing it with $(x^k, y^k)$, which is randomly selected from previous SVRG subroutine. This is partly inspired by the GD-SVRG algorithm for minimizing PL functions \citep{reddi2016stochastic}. Notice when $T=1$, VR-AGDA reduces to a double-loop algorithm which is similar to the  SVRG for saddle point problems proposed by \citet{palaniappan2016stochastic}, except for several notable differences: (i) we are using the alternating updates rather than simultaneous updates,  (ii) as a result, we require to sample two independent indices rather than one at each iteration, and (iii) most importantly, we are dealing with possibly nonconvex-nonconcave objectives that satisfy the two-sided PL condition. 




The following two theorems capture the convergence of VR-AGDA. 

\begin{theorem} \label{svrg thm 2}
Suppose Assumptions \ref{Existence of global minimax point} and \ref{Lipscthitz gradient 2} hold and $f(x, y)$ satisfies the two-sided PL condition with $\mu_1$ and $\mu_2$. Define $P_k = a^k + \frac{1}{20}b^k$, where $a^k = \mathbb{E}[g(x^k)-g^*]$ and $b^k = \mathbb{E}[g(x^k) - f(x^k, y^k)]$. If we run VR-AGDA with $\tau_1 =  \beta/(28\kappa^8l)$, $\tau_2 = \beta/(l\kappa^6)$, $N =\lfloor \alpha\beta^{-2/3}\kappa^{9}(2+4\beta^{1/2}\kappa^{-3})^{-1}\rfloor$ and $T=1$, where $\alpha, \beta$ are constants irrelevant to $l,n, \mu_1, \mu_2$,  then
    $P_{k+1} \leq \frac{1}{2}P_k.$
This  further implies a total complexity of $$\mathcal{O}\big((n+\kappa^9)\log(1/\epsilon)\big)$$ for VR-AGDA to achieve an $\epsilon$-optimal solution.
\end{theorem}

\begin{theorem} \label{svrg thm 1}
Under the same assumptions in Theorem \ref{svrg thm 2} and further assuming $n\leq \kappa^9$ , if we run VR-AGDA with $\tau_1 =  \beta/(28\kappa^2ln^{2/3})$, $\tau_2 = \beta/(ln^{2/3})$, $N =\lfloor \alpha\beta^{-2/3}n(2+4\beta^{1/2}n^{-1/3})^{-1}\rfloor$, and $T = \lceil\kappa^3n^{-1/3}\rceil$, where $\alpha, \beta$ are constants irrelevant to $l,n, \mu_1, \mu_2$,  then
    $P_{k+1} \leq \frac{1}{2}P_k.$
This further implies a total  complexity of $$\mathcal{O}\big(n^{2/3}\kappa^3\log(1/\epsilon)\big)$$
for VR-AGDA to achieve an $\epsilon$-optimal solution.
\end{theorem}

\begin{figure}[!ht]
\begin{center}
\centerline{\includegraphics[width=0.4\columnwidth]{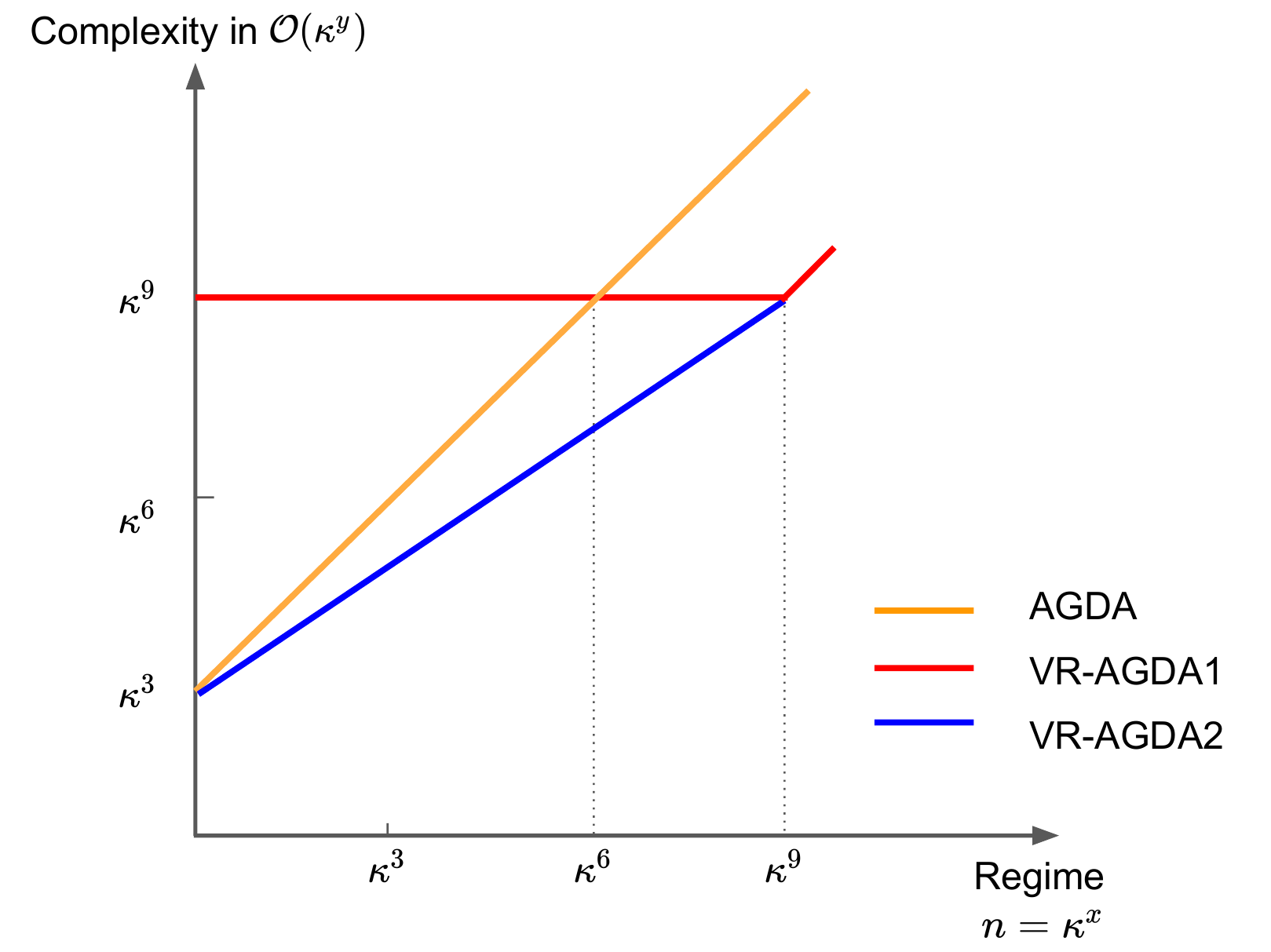}}
\caption{Comparison of complexities of AGDA and VR-AGDA, where VR-AGDA1, VR-AGDA2 correspond to the two settings in Theorems \ref{svrg thm 2} and \ref{svrg thm 1}. In the regime $n\leq \kappa^9$, VR-AGDA2 performs best; in the regime $n\geq \kappa^9$, VR-AGDA1 performs best. }
\label{complexity figure}
\end{center}
\vskip -0.2in
\end{figure}

\begin{remark}
 Theorems \ref{svrg thm 2} and \ref{svrg thm 1} are different in their choices of stepsizes and iteration numbers, which gives rise to different complexities. Another difference is that Theorem \ref{svrg thm 1} only works in the regime where the number of components $n$ is not ``too large" compared to the condition number, i.e., $n\leq \kappa^9$, which naturally guarantees $T = \lceil\kappa^3n^{-1/3}\rceil \geq 1$. 
\end{remark}

\begin{remark}
Since AGDA has complexity $\mathcal{O}\big( n\kappa^3\log(1/\epsilon) \big)$, VR-AGDA with the setting in Theorem \ref{svrg thm 2} is better than AGDA when $n \geq\mathcal\kappa^6$. With the setting in Theorem \ref{svrg thm 1}, VR-AGDA outperforms AGDA as long as the  assumption $n\leq \kappa^9$ holds. As a result of these two theorems, VR-AGDA always improves over AGDA. Furthermore, VR-AGDA with the second setting has a lower complexity than the first setting in the regime $n\leq \kappa^9$, although the first setting allows a simpler double-loop algorithm. Figure \ref{complexity figure} summarizes the performance of VR-AGDA compared to AGDA in different regimes of $n$ and $\kappa$. 
\end{remark}
\vspace{-0.45cm}


%% file: Experiments.tex
\section{Experiments}
\label{Sec5}

\noindent In the introduction, we already presented the convergence results of AGDA on a two-dimensional nonconvex-nonconcave function that satisfies the two-sided PL condition. In this section, we will present numerical experiments on machine learning applications: robust least square and imitation learning for linear quadratic regulators (LQR). Particularly, we focus on the comparison between AGDA, Stoc-AGDA, and VR-AGDA.

\subsection{Robust least square}

\begin{figure*} [!t]
\centering
\subfigure[Dataset 1]{%
\label{fig:first}%
\includegraphics[height=1.63in]{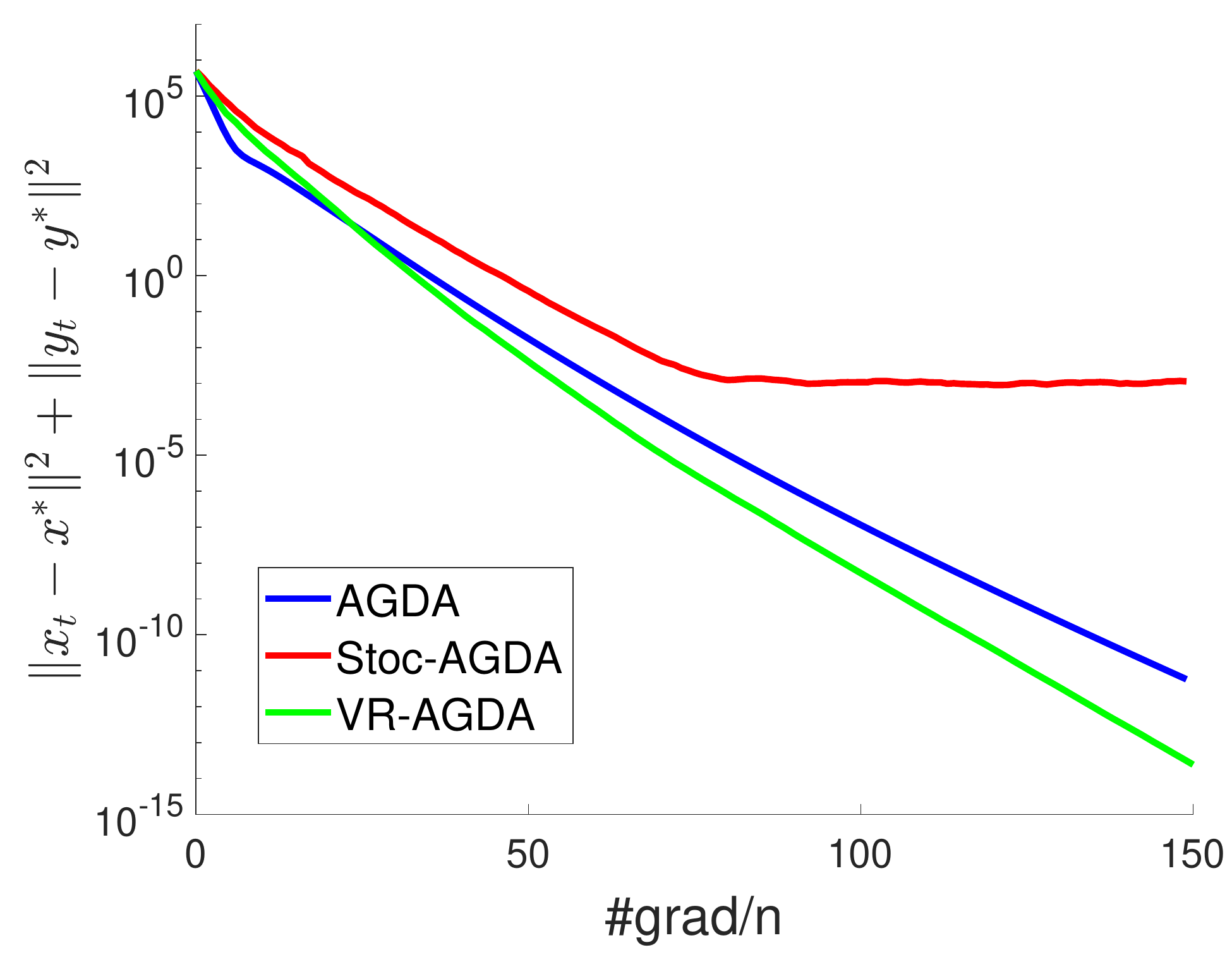}}%
\subfigure[Dataset 2]{%
\label{fig:second}%
\includegraphics[height=1.63in]{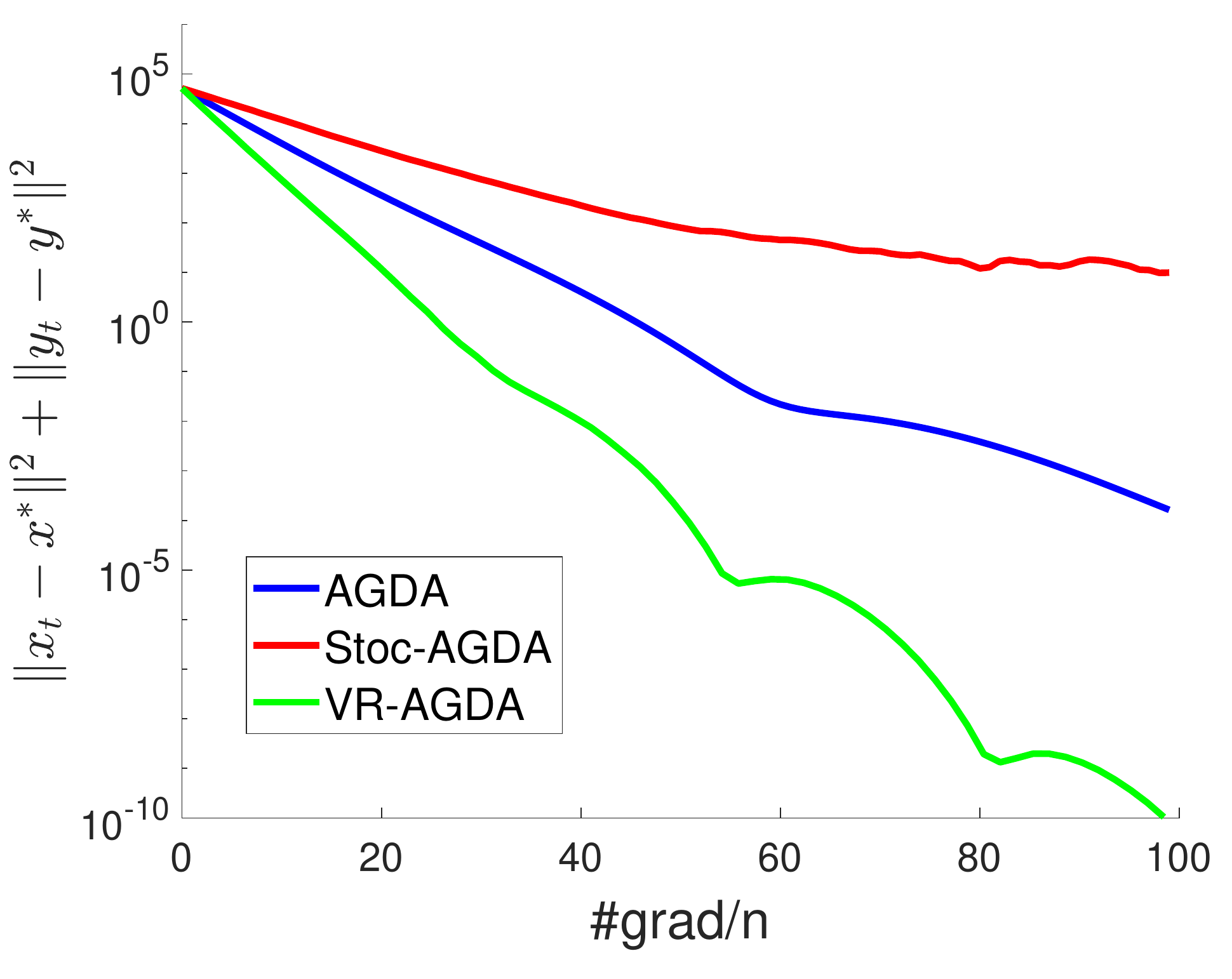}}%
\subfigure[Dataset 3]{%
\label{fig:second}%
\includegraphics[height=1.63in]{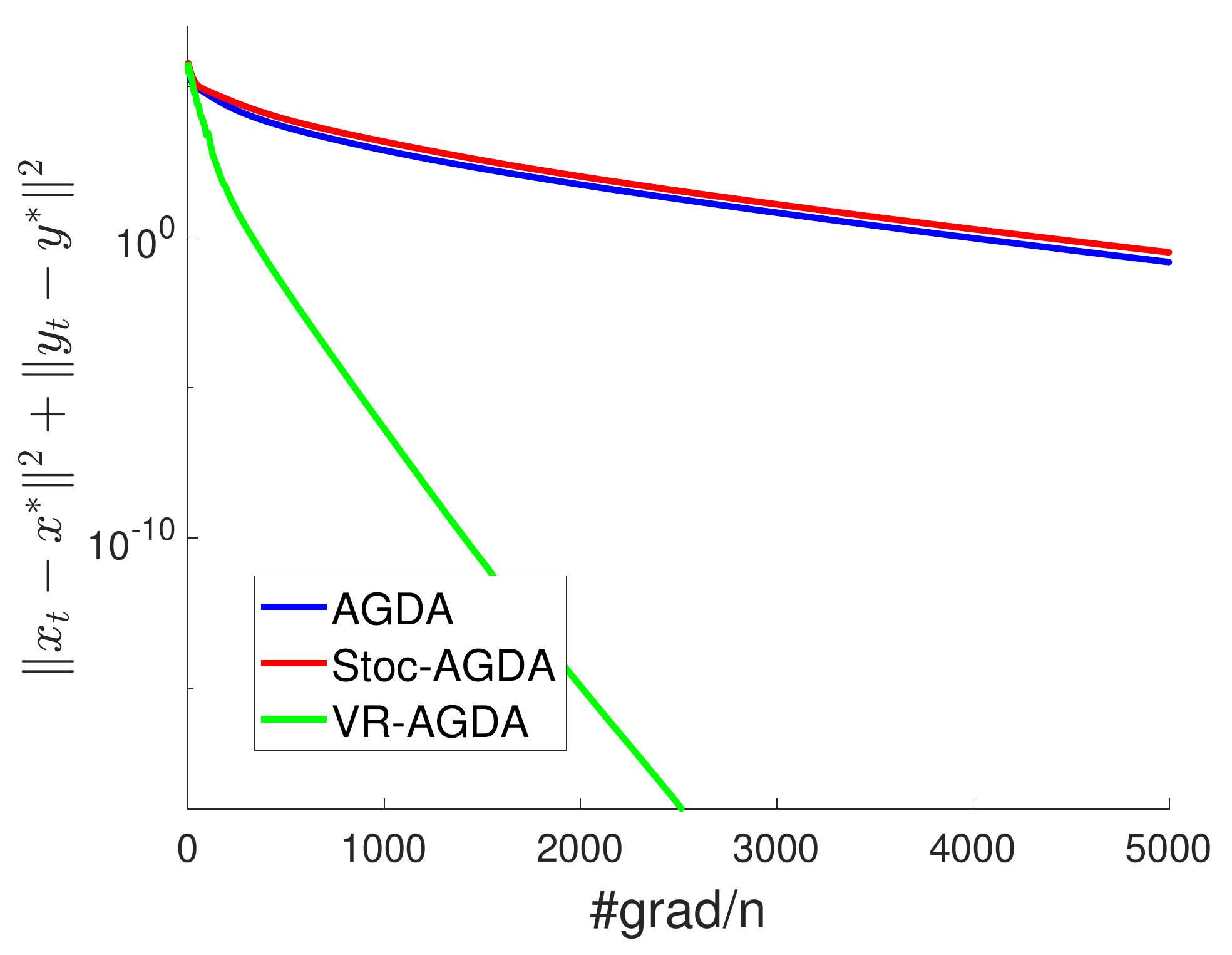}}%
\vspace{-0.4cm}

\subfigure[Dataset 1]{%
\label{fig:third}%
\includegraphics[height=1.63in]{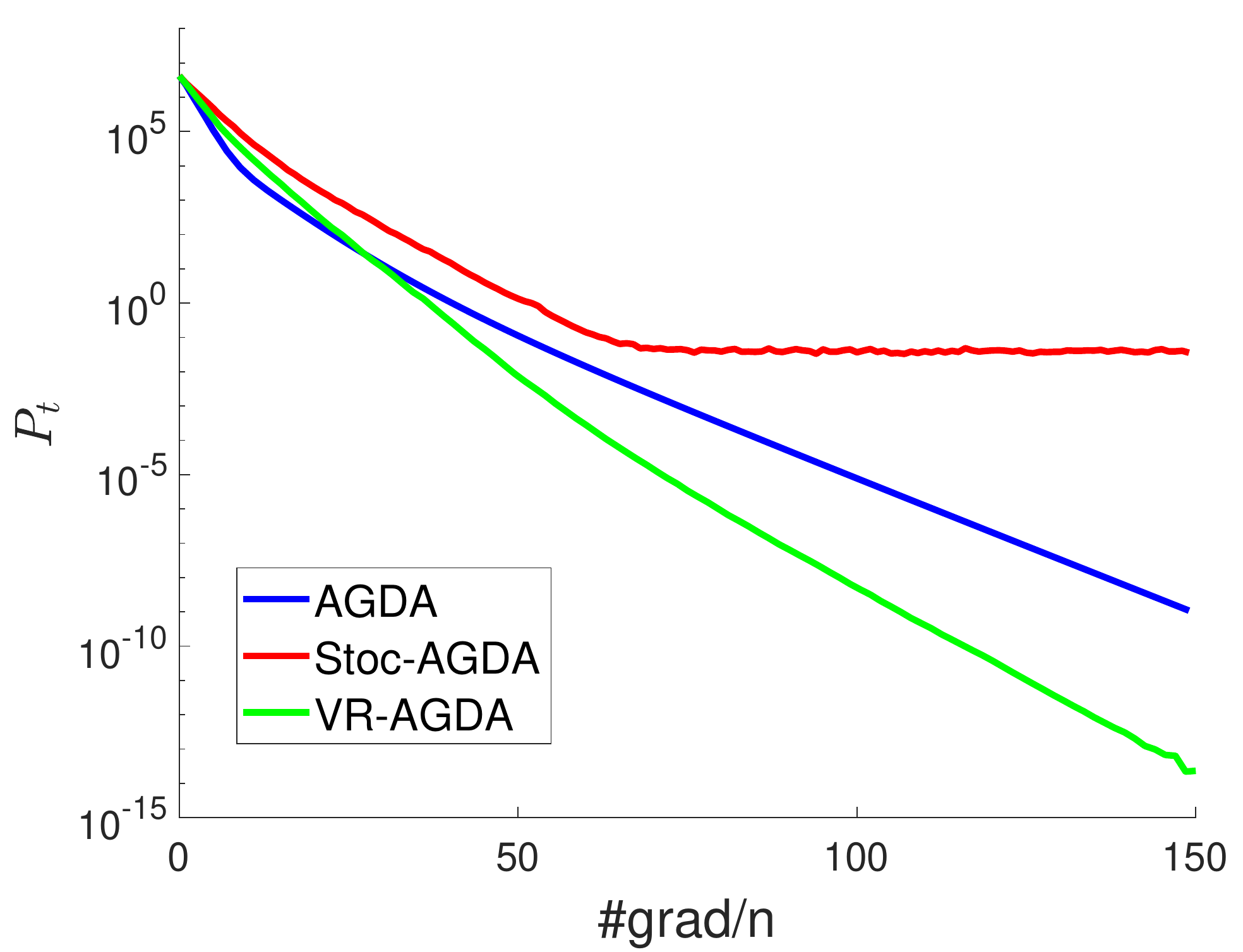}}%
\subfigure[Dataset 2]{%
\label{fig:third}%
\includegraphics[height=1.63in]{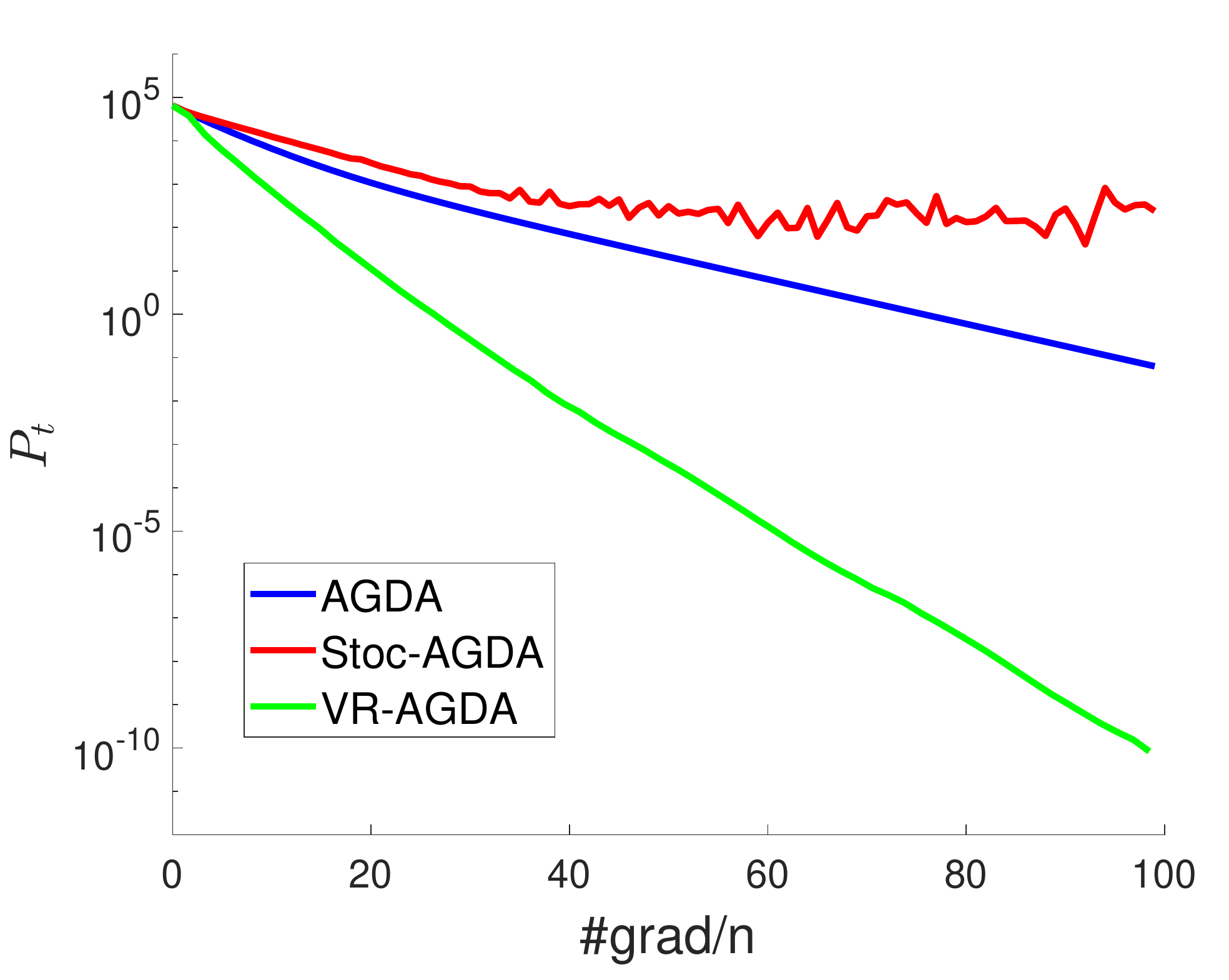}}%
\subfigure[Dataset 3]{%
\label{fig:fourth}%
\includegraphics[height=1.63in]{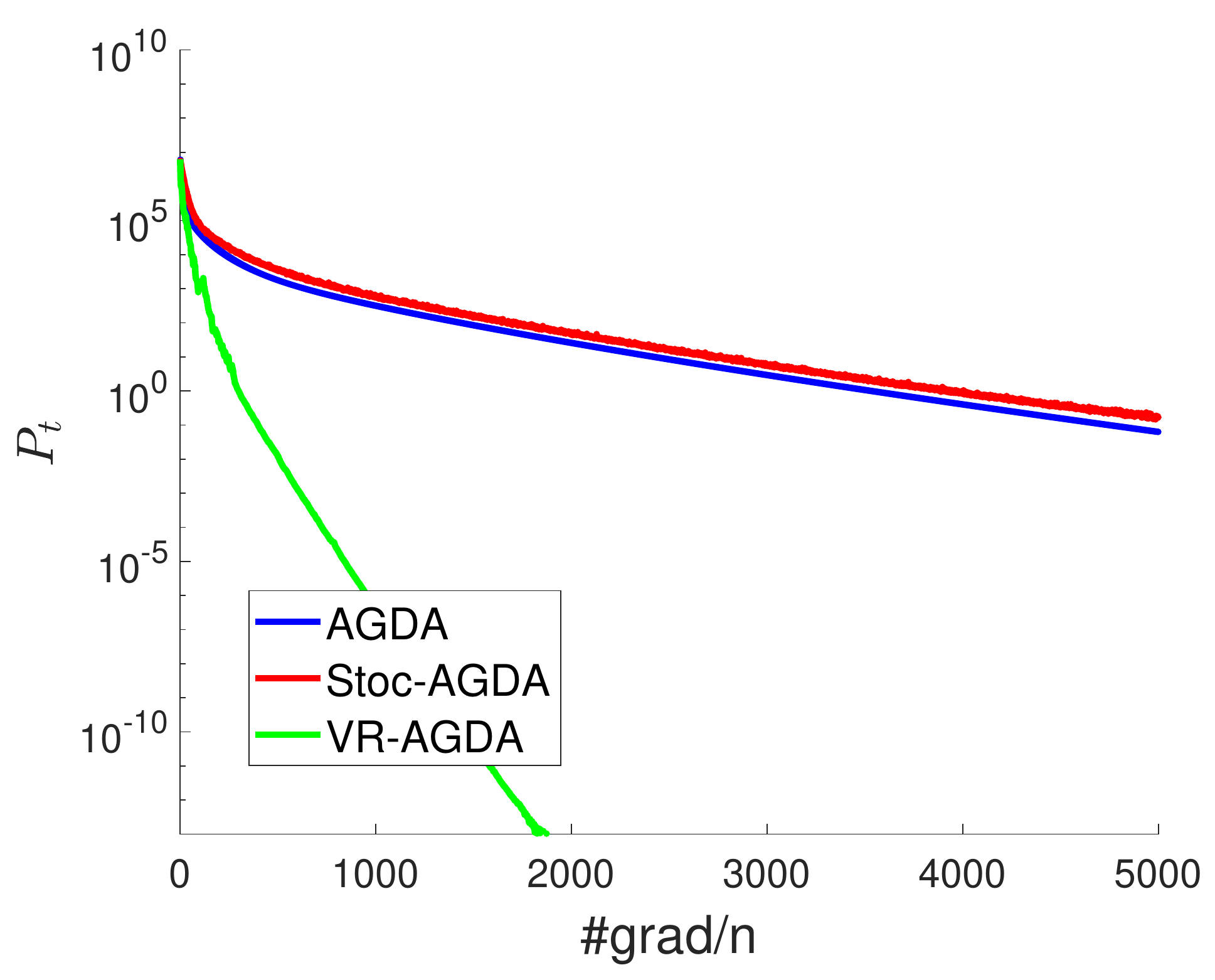}}
\vspace{-0.4cm}
\caption{Comparison of the convergences of AGDA, Stoc-AGDA and SVRG-AGDA on three datasets based on two inaccuracy measures: (i) $\Vert x_t - x^*\Vert^2 + \Vert y_t-y^*\Vert^2$ ( as shown in the first row), and (ii) $P_t = (g(x_t) - g^*) + (g(x_t) - f(x_t, y_t))$ (as shown in the second row). } \label{figure 2}
\end{figure*}

\noindent We consider the least square problems with coefficient matrix $A\in \mathbb{R}^{n\times m}$ and noisy vector $y_0\in \mathbb{R}^n$. We assume that $y_0$ is subject to bounded deterministic perturbation $\delta$. Robust least square (RLS) minimizes the worst case residual, and can be formulated as \citep{el1997robust}:
\begin{equation*}
    \min_x \max_{\delta:\Vert \delta \Vert \leq \rho} \Vert Ax - y\Vert^2, \text{ where } \delta = y_0 -y.
\end{equation*}
We consider RLS with soft constraint:
\begin{equation} \label{RLS}
    \min_x \max_y F(x, y):=\Vert Ax - y\Vert_M^2  - \lambda \Vert y- y_0\Vert_M^2,
\end{equation}
where we also adopt the general M-(semi-)norm in (\ref{RLS}): $\Vert x \Vert_M^2 = x^TMx$ and $M$ is positive semi-definite. $F(x, y)$ satisfies the two-sided PL condition when $\lambda>1$, because it can be written as the composition of a strongly-convex-strongly-concave function and an affine function (Example \ref{ex 2}). However, $F(x,y)$ is not strongly convex about $x$, and when $M$ is not full-rank, it is not strongly concave about $y$. 

\textbf{Datasets.} We use three datasets in the experiments, and two of them are generated in the same way as in \citet{du2019linear}. We generate the first dataset with $n=1000$ and $m=500$ by sampling rows of $A$ from a Gaussian  $\mathcal{N}(0, I_n)$ distribution and setting $y_0 = Ax^*+\epsilon$ with $x^*$ from Gaussian $\mathcal{N}(0, 1)$  and $\epsilon$ from Gaussian $\mathcal{N}(0, 0.01)$. We set $M = I_n$ and $\lambda = 3$. The second dataset is the rescaled aquatic toxicity dataset by \citet{cassotti2014prediction}, which uses 8 molecular descriptors of 546 chemicals  to predict quantitative acute aquatic toxicity towards Daphnia Magna. We use $M = I$ and $\lambda = 2$ for this dataset. The third dataset is generated with $A\in \mathbb{R}^{1000\times 500}$ from Gaussian $\mathcal{N}(0, \Sigma)$ where $\Sigma_{i,j} = 2^{-|i-j|/10}$, $M$ being rank-deficit with positive eigenvalues sampled from $[0.2, 1.8]$ and $\lambda = 1.5$. These three datasets represent cases with low, median, and high condition numbers, respectively.

\textbf{Evaluation.} For each dataset, we compare three algorithms: AGDA, Stoc-AGDA, and VR-AGDA. We tune the stepsizes of all algorithms to achieve the best convergence. For Stoc-AGDA, we choose constant stepsizes to form a fair comparison with the other two. We report the potential function value, i.e., $P_t$ described in our theorems, and distance to the limit point $\Vert (x_t, y_t) - (x^*, y^*)\Vert^2$. These errors are plotted against the number of gradient evaluations normalized by $n$ (i.e., number of full gradients). Results are reported in Figure \ref{figure 2}. We observe that VR-AGDA and AGDA both exhibit linear convergence, and the speedup of VR-AGDA is fairly significant when the condition number is large, whereas Stoc-AGDA progresses fast at the beginning and stagnates later on. These numerical results clearly validate our theoretical findings.

\subsection{Generative adversarial imitation learning for LQR} \label{lqr section}

\begin{figure*} [!ht]  
\centering
\subfigure[$d=3, k=2$]{%
\label{fig:first}%
\includegraphics[height=0.27\linewidth]{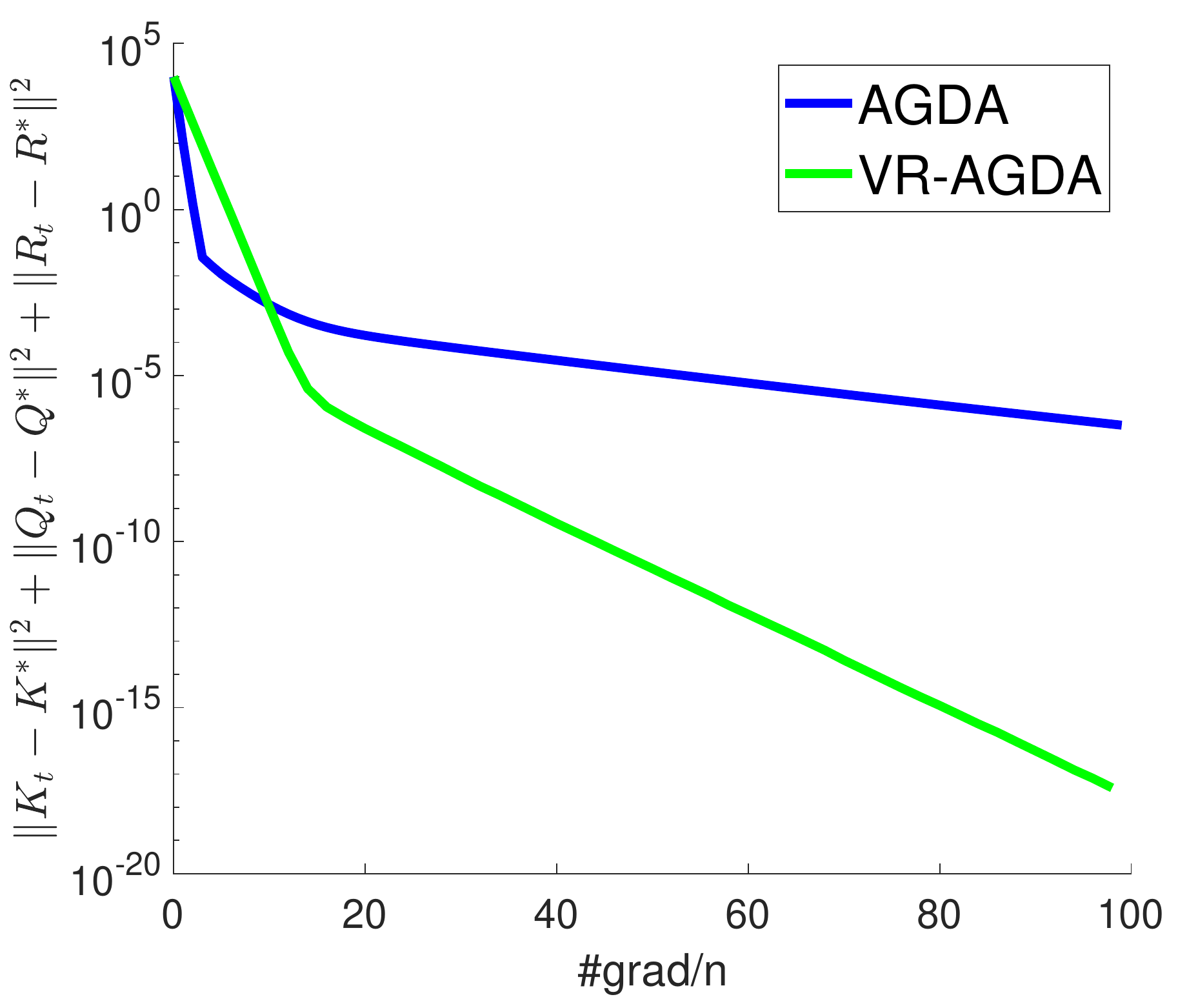}}%
\subfigure[$d=20, k=10$]{%
\label{fig:third}%
\includegraphics[height=0.27\linewidth]{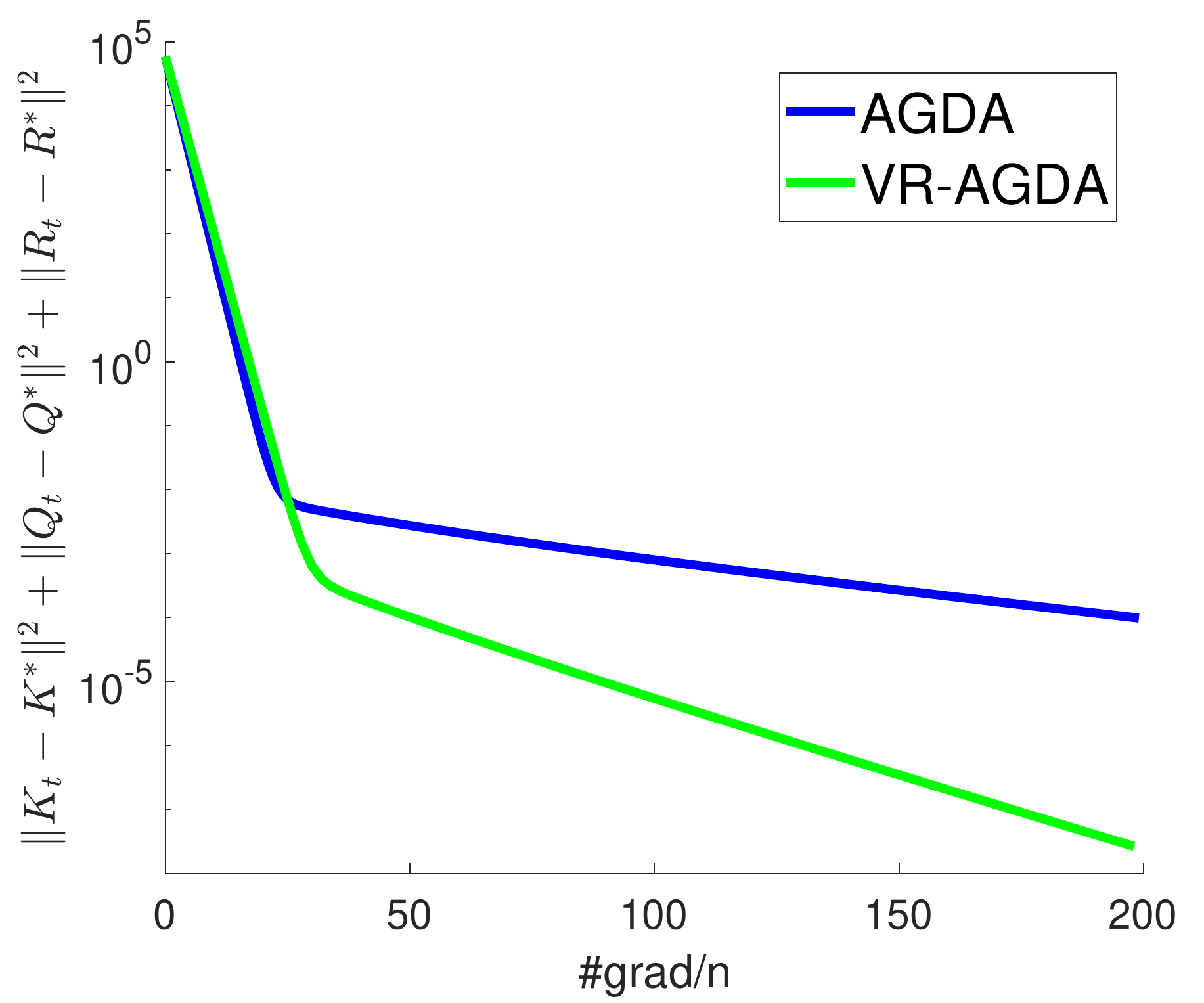}}%
\subfigure[$d=30, k=20$]{%
\label{fig:fourth}%
\includegraphics[height=0.27\linewidth]{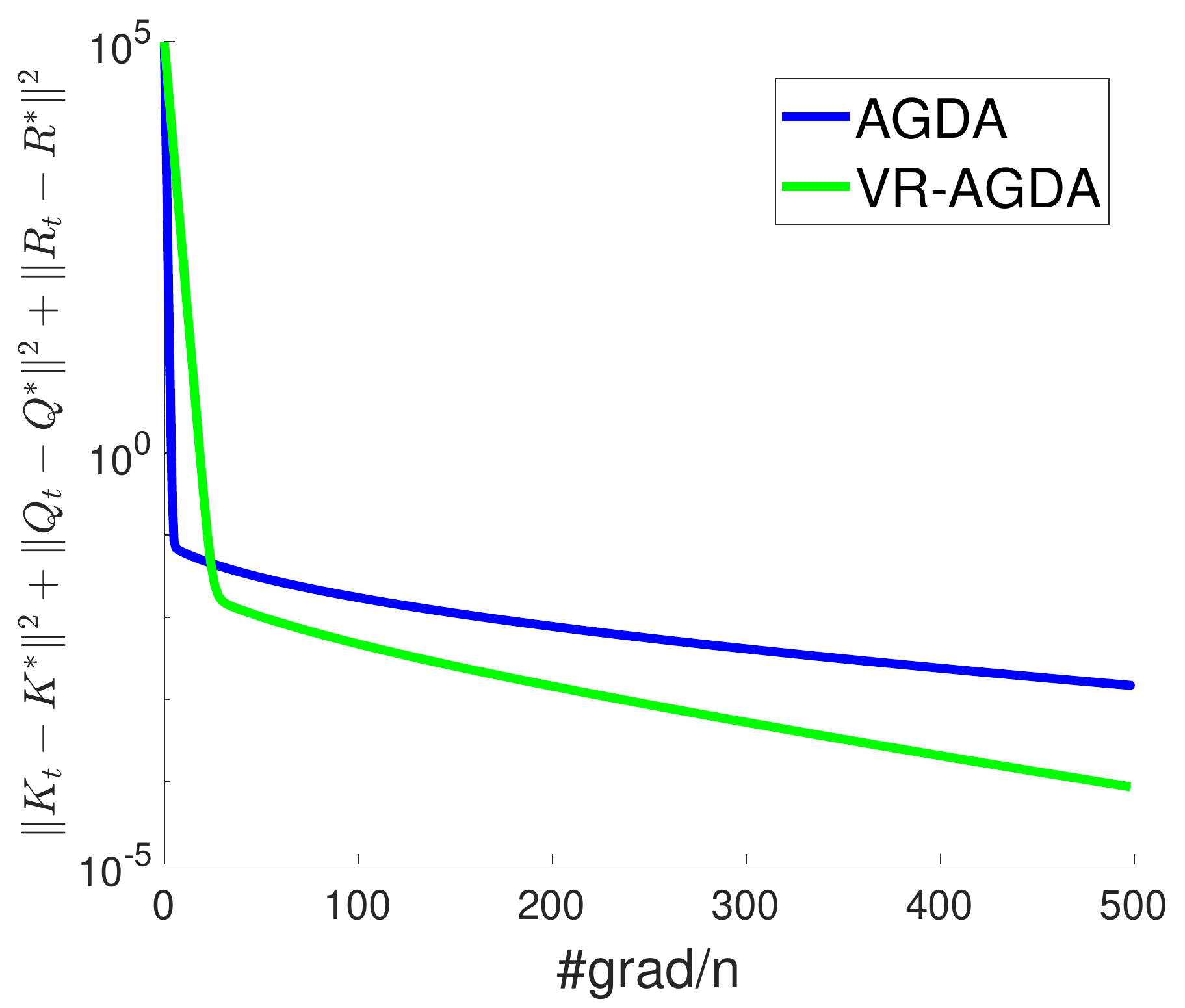}}
\vspace{-0.4cm}
\caption{AGDA and VR-AGDA on generative adversarial learning for LQR} \label{figure lqr}
\end{figure*}

\noindent The optimal control problem for LQR can be formulated as:
\begin{align*}
    \minimize_{\pi_t} \quad &\mathbb{E}_{x_0\sim \mathcal{D}} \sum_{t=0}^{\infty} x_{t}^{\top} Q x_{t}+u_{t}^{\top} R u_{t} \\
    \text{such that}\quad  &x_{t+1} = Ax_t + Bu_t,  u_t = \pi_t(x_t)
\end{align*}
where $x_t \in \mathbb{R}^d$ is a state,  $u_t \in \mathbb{R}^k$ is a control, $\mathcal{D}$ is the distribution of initial state $x_0$, and $\pi_t$ is a policy. It is known that the optimal policy is linear: $u_t = -K^*x_t$, where $K^*\in \mathbb{R}^{k\times d}$. If we parametrize the policy in the linear form, $u_t = -Kx_t$, the problem can be written as:
\begin{equation*}
    \min_K C(K; Q,R):=\mathbb{E}_{x_{0} \sim \mathcal{D}}\left[\sum_{t=0}^{\infty}\left(x_{t}^{\top} Q x_{t}+u_{t}^{\top} R u_{t}\right)\right]
\end{equation*}
where the trajectory is induced by LQR dynamics and policy $K$. In generative adversarial imitation learning for LQR, the trajectories induced by an expert policy $K_E$ are observed and part of the goal is to learn the cost function parameters $Q$ and $R$ from the expert. This can be formulated as a minimax problem \citep{cai2019global}:
\begin{equation*}
    \min_K \max_{(Q, R)\in \Theta} m(K,Q,R)
\end{equation*}
where $m(K,Q,R):= C(K;Q,R)-C(K_E;Q,R)-\Phi(Q,R)$, $\Theta = \{(Q, R): \alpha_{Q} I \preceq Q \preceq \beta_{Q} I, \quad \alpha_{R} I \preceq R \preceq \beta_{R} I \}$ and $\Phi$ is a strongly-convex regularizer. We sample $n$ initial points $x_0^{(1)}, x_0^{(2)},..., x_0^{(n)}$ from $\mathcal{D}$ and approximate $C(K; Q,R)$ by sample average 
\begin{equation*}
    C_n(K;Q,R) := \frac{1}{n}\sum_{i=1}^n\left[\sum_{t=0}^{\infty}\left(x_{t}^{\top} Q x_{t}+u_{t}^{\top} R u_{t}\right)\right]_{x_0 = x_0^{(i)}}.
\end{equation*}
We then consider 
\begin{equation*}
    m_n(K,Q,R) =  C_n(K;Q,R)-C_n(K_E;Q,R)-\Phi(Q,R).
\end{equation*}
Note that $m_n$ satisfies the PL condition in terms of $K$ \citep{fazel2018global}, and $m_n$ is strongly-concave in terms of $(Q,R)$, so the function satisfies the two-sided PL condition. 

In our experiment, we use $\Phi(Q,R) = \lambda (\Vert Q-\Bar{Q}\Vert^2 + \Vert R-\Bar{R}\Vert^2)$ for some $\Bar{Q}, \Bar{R}$ and $\lambda = 1$. We generate three datasets with different dimensions: (1) $d=3, k=2$; (2) $d=20, k=10$; (3) $d = 30, k=20$. The initial distribution $\mathcal{D}$ is $\mathcal{N}(0, I_d)$ and we sample $n=100$ initial points. 
The exact gradients can be computed based on the compact forms established in \citet{fazel2018global, cai2019global}. We compare AGDA and VR-AGDA under fine-tuned stepsizes, and track their errors in terms of $\Vert K_t -K^*\Vert^2 + \Vert Q_t - Q^*\Vert_F^2+\Vert R_t - R^*\Vert_F^2 $.  The result is reported in Figure \ref{figure lqr}, which again indicates that VR-AGDA significantly outperforms AGDA.  

%% file: appendix1.tex
\noindent  {\Large \bf Appendix}
\section{Proofs for Section \ref{Sec2}}
\label{appendix1}

We first present several key lemmas.

\begin{lemma}[\citet{karimi2016linear}]  \label{PL to EB QG}
If $f(\cdot)$ is l-smooth and it satisfies PL with constant $\mu$, then it also satisfies error bound (EB) condition with $\mu$, i.e.
\begin{equation*}
    \Vert \nabla f(x)\Vert \geq \mu \Vert x_p-x\Vert, \forall x,
\end{equation*}
where $x_p$ is the projection of $x$ onto the optimal set, also it satisfies quadratic growth (QG) condition with $\mu$, i.e.
\begin{equation*}
     f(x)-f^*\geq \frac{\mu}{2}\Vert x_p-x\Vert^2, \forall x.
\end{equation*}
Conversely, if $f(\cdot)$ is l-smooth and it satisfies EB with constant $\mu$, then it satisfies PL with constant $\mu/l$. 
\end{lemma}

From the above lemma, we easily derive that $l\geq \mu$.

\begin{lemma} [\citet{nouiehed2019solving}]  \label{g smooth}
In the minimax problem, when $-f(x,\cdot)$ satisfies PL condition with constant $\mu_2$ for any $x$ and $f$ satisfies Assumption \ref{Lipscthitz gradient}, then the function $g(x) := \max_y f(x,y)$ is $L$-smooth with $L:=l+l^2/\mu_2$ and $\nabla g(x) = \nabla_x f(x, y^*(x))$ for any $y^*(x)\in \arg\max_y f(x,y)$. 
\end{lemma}

\begin{lemma}  \label{g PL}
In the minimax problem \ref{objective}, when the objective function $f$ satisfies Assumption \ref{Lipscthitz gradient} (Lipschitz gradient) and the two-sided PL condition with constant $\mu_1$ and $\mu_2$, then function $g( x ) := \max_{ y }f( x ,  y )$ satisfies the PL condition with $\mu_1$. 
\end{lemma}

\begin{proof}
From Lemma \ref{g smooth}, 
\begin{equation*}
    \Vert \nabla g( x )\Vert^2 = \Vert \nabla_xf( x ,  y ^*( x ))\Vert^2.
\end{equation*}
Since $f(\cdot,  y )$ satisfies PL condition with constant $\mu_1$, we get
\begin{equation}
    \Vert \nabla g( x )\Vert^2 \geq 2\mu_1[f( x ,  y ^*( x )) - \min_{x'}f( x ', y ^*( x ))]. \label{g_PL1}
\end{equation}
Also, 
\begin{equation}
    f( x ', y ^*( x ))\leq \max_y f( x ', y ) \Longrightarrow  \min_{x'}f( x ', y ^*( x ))\leq \min_{x'}\max_y f( x ', y ) = g^*. \label{g-PL2}
\end{equation}
Combining equation (\ref{g_PL1}) and (\ref{g-PL2}), we obtain, 
\begin{equation*}
    \Vert \nabla g( x )\Vert^2 \geq 2\mu_1 (g( x )- g^*).
\end{equation*}
\end{proof}

The following lemma states that stochastic gradient descent converges linearly to the neighbourhood of the optimal set under PL condition. The proof is based on \citep{karimi2016linear}.

\begin{lemma}  \label{min linear}
Consider the optimization problem $\min_x f(x)=\mathbb{E}[F(x;\xi)]$, where $f$ is $l$-smooth and satisfies PL condition with constant $\mu$. Using the stochastic gradient descent with step size $\tau \leq 1/l$,
\begin{equation*}
    x_{t+1} = x_t -\tau G(x_t,\xi_t),
\end{equation*}
where 
\begin{equation*}
    \mathbb{E}[G(x,\xi)-\nabla f(x)] = 0, \qquad \mathbb{E}[\Vert G(x,\xi)-\nabla f(x)\Vert^2]\leq \sigma^2,
\end{equation*}
then we have 
\begin{equation*}
    \mathbb{E}[f(x_{t+1})-f^*]\leq (1-\mu\tau)\mathbb{E}[f(x_{t})-f^*]+\frac{l\tau^2}{2}\sigma^2.
\end{equation*} \label{A9}
\end{lemma}

\begin{proof}
By smoothness of $f$ we have
\begin{align*}
    f(x_{t+1})-f^* &\leq f(x_t) + \langle \nabla f(x_t), x_{t+1}-x_t\rangle +\frac{l}{2}\Vert x_{t+1}-x\Vert^2 -f^*\\
    &= f(x_t) - \tau\langle \nabla f(x_t), G(x_t,\xi_t)\rangle +\frac{l\tau^2}{2}\Vert G(x_t,\xi_t)\Vert^2 -f^*.
\end{align*}
Taking expectation of both sides, we get
\begin{align} \nonumber
    \mathbb{E}[f(x_{t+1})-f^*] \leq& \mathbb{E}[f(x_t)-f^*] -\tau\mathbb{E}[\Vert\nabla f(x_t)\Vert^2] + \frac{l\tau^2}{2} \mathbb{E}[\Vert G(x_t,\xi_t)\Vert^2] \\ \nonumber
    = & \mathbb{E}[f(x_t)-f^*] -\tau\mathbb{E}[\Vert\nabla f(x_t)\Vert^2] + \frac{l\tau^2}{2} \mathbb{E}[\Vert \nabla f(x_t)\Vert^2] \\ \nonumber
    &\qquad+ \frac{l\tau^2}{2}\mathbb{E}[\Vert \nabla f(x_t)-G(x_t,\xi_t)\Vert^2] \\ \nonumber
    \leq& \mathbb{E}[f(x_t)-f^*] -\frac{\tau}{2}\mathbb{E}[\Vert\nabla f(x_t)\Vert^2] +\frac{l\tau^2}{2}\sigma^2 \\ \nonumber
    \leq& (1-\mu\tau)\mathbb{E}[f(x_t)-f^*] + \frac{l\tau^2}{2}\sigma^2,
\end{align}
where in the equality we use $ \mathbb{E}[ G(x_t,\xi_t)] = \nabla f(x_t)$, in the second inequality we use $\tau\leq 1/l$, and we use PL condition in the last inequality.
\end{proof}

\textbf{Proof for Lemma \ref{equivalent optimality}}.
\begin{proof}
\begin{itemize}
    \item (stationary point) $\Longrightarrow$ (saddle point): From the definition of PL condition, if $( x ^*, y ^*) $ is a stationary point,
    \begin{align*}
    \max_{y}f( x ^*, y )-f( x ^*, y ^*)\leq \frac{1}{2\mu_2}\left\Vert \nabla_yf( x ^*, y ^*)\right\Vert^2 = 0,\\
     f( x ^*, y ^*)-\min_{ x }f( x ,  y ^*)\leq \frac{1}{2\mu_1}\left\Vert \nabla_xf( x ^*, y ^*)\right\Vert^2 = 0,
\end{align*}
so $\max_{ y }f( x ^*, y )=f( x ^*, y ^*)=\min_{ x }f( x ,  y ^*)$, and therefore $f( x ^*,  y ^*)$ is a saddle point.
\item (saddle point) $\Longrightarrow$ (global minimax point): Follow from definitions.
\item (global minimax point) $\Longrightarrow$ (stationary point): If $( x ^*, y ^*) $ is a global minimax point, then by definition, \begin{equation*}
     y ^* \in\arg\max_y f( x ^*,  y ^*),
     x ^* \in\arg\min_x g( x ),
\end{equation*}
Then by first order necessary condition, we have,
\begin{equation*}
    \nabla_y f( x ^*,  y ^*) = 0,\\
    \nabla g( x ^*) = 0,
\end{equation*}
Further with Lemma \ref{g smooth}, 
\begin{equation*}
    \nabla g( x ^*) = \nabla_x f( x ^*,  y ^*) = 0
\end{equation*}
Thus, $( x ^*, y ^*)$ is a stationary point. 
\end{itemize}
\end{proof}

\begin{prop}
The function
\begin{equation*}
    f(x,y) = x^2 + 3\sin^2x\sin^2y-4y^2-10\sin^2y, 
\end{equation*}
satisfies the two-sided PL condition with $\mu_1 = 1/16, \mu_2 = 1/14$.
\end{prop}

\begin{proof}
It is not hard to derive that $\arg\min_x f(x,y)=0, \forall y$, and $\arg\max_y f(x,y) = 0, \forall x$, i.e. $x^*(y) = y^*(x) = 0, \forall x,y$. Therefore, $(0,0)$ is the only saddle point. Then compute the gradients:
\begin{align*}
    &\nabla_xf(x,y)=2x+3\sin^2(y)\sin(2x),\\
    &\nabla_yf(x,y)=-8y+3\sin^2(x)\sin(2y)-10\sin(2y).
\end{align*}
and 
\begin{align*}
    &\vert\nabla^2_xf(x,y)\vert = \vert 2 + 6\sin^2(y) \cos(2x)\vert \leq 8,\\
    &\vert\nabla^2_yf(x,y)\vert = \vert -8 + 6\sin^2(x)\cos(2y) - 20\cos(2y)\vert \leq 28.
\end{align*}
so $f(\cdot, y)$ is $L_1$-smooth with $L_1 = 8$ for any $x$ and $f(x, \cdot)$ is $L_2$-smooth with $L_2 = 28$ for any $y$. Then note that:
\begin{align*}
    &\frac{\vert \nabla_x f(x,y) \vert}{\vert x - x^*(y)\vert} = \frac{\vert \nabla_xf(x,y)\vert}{\vert x\vert} = \frac{\vert 2x+3\sin^2(y)\sin(2x)\vert }{\vert x\vert}\geq \frac{1}{2},\\
    &\frac{\vert \nabla_y f(x,y) \vert}{\vert y - y^*(x)\vert} = \frac{\vert \nabla_y f(x,y)\vert}{\vert y\vert} = \frac{\vert -8y+3\sin^2(x)\sin(2y)-10\sin(2y)\vert }{\vert y\vert}\geq 2
\end{align*}
So $f(\cdot,y)$ satisfies EB with $\mu_{EB1} = 1/2$, and -$f(x,\cdot)$ satisfies EB with $\mu_{EB2} = 2$. By Lemma \ref{PL to EB QG}, we have $f(\cdot,y)$ satisfies PL with constant $\mu_1 = 1/16$ and -$f(x,\cdot)$ satisfies PL with constant $\mu_1 = 1/14$. 

\end{proof}

%% file: appendix2.tex
\section{Proofs for Section \ref{Sec3}}
\label{appendix2}

Before we step into proofs for Theorem \ref{main}, \ref{main deterministic} and \ref{main diminishing}, we first present a contraction theorem for each iteration. 

\begin{theorem} \label{contraction theorem}
Assume Assumption \ref{Lipscthitz gradient}, \ref{Existence of global minimax point}, \ref{stochastic gradients} hold and $f(x, y)$ satisfies the two-sided PL condition with $\mu_1$ and $\mu_2$. Define $a_t = \mathbb{E}[g(x_t)-g^*]$ and $b_t = \mathbb{E}[g(x_t) - f(x_t, y_t)]$. If we run one iteration of Algorithm \ref{s-agda} with $\tau_1^t = \tau_1 \leq 1/L$ ($L$ is specified in Lemma \ref{g smooth}) and $\tau_2^t = \tau_2\leq 1/l$, then
\begin{equation*}
    a_{t+1} + \lambda b_{t+1} \leq \max\{k_1, k_2\}(a_t+\lambda b_t) +  \lambda(1-\mu_2\tau_2)\frac{L+l}{2}\tau_1^2\sigma^2 + \frac{l}{2}\lambda\tau_2^2\sigma^2 + \frac{L}{2}\tau_1^2\sigma^2,
\end{equation*}
where 
\begin{align} \label{k1}
    &k_1 := 1-\mu_1\big[\tau_1 + \lambda(1-\mu_2\tau_2)\tau_1 - \lambda(1+\beta)(1-\mu_2\tau_2)(2\tau_1 + l\tau_1^2)\big], \\ \label{k2}
    & k_2 := 1-\mu_2\tau_2 + \frac{l^2\tau_1}{\mu_2\lambda} + (1-\mu_2\tau_2)\frac{l^2}{\mu_2}\tau_1 + (1+\frac{1}{\beta})(1-\mu_2\tau_2)\frac{l^2}{\mu_2}(2\tau_1 + l\tau_1^2),
\end{align}
and $\lambda, \beta >0 $ such that $k_1 \leq 1$. 
\end{theorem}

\begin{proof}
Because $g$ is $L$-smooth by Lemma \ref{g smooth}, we have 
\begin{align*}
    g(x_{t+1})-g^* \leq& g(x_t)-g^*+\langle \nabla g(x_t), x_{t+1}-x_t \rangle +\frac{L}{2}\Vert x_{t+1}-x_t\Vert^2\\
     =& g(x_t)-g^* - \tau_1\langle \nabla g(x_t), G_x(x_t, y_t,\xi_{t1}) \rangle +\frac{L}{2}\tau_1^2\Vert G_x(x_t, y_t,\xi_{t1}) \Vert^2.
\end{align*}
Taking expectation of both side and use Assumption \ref{stochastic gradients}, we get
\begin{align}  \nonumber
    \mathbb{E}[ g(x_{t+1})-g^*] \leq &\mathbb{E}[g(x_t)-g^*] - \tau_1\mathbb{E}[\langle \nabla g(x_t), \nabla_xf(x_t,y_t) \rangle] + \frac{L}{2}\tau_1^2\mathbb{E}[\Vert G_x(x_t, y_t,\xi_{t1}) \Vert^2] \\ \nonumber
    \leq& \mathbb{E}[g(x_t)-g^*] - \tau_1\mathbb{E}[\langle \nabla g(x_t), \nabla_xf(x_t,y_t) \rangle] +\frac{L}{2}\tau_1^2\mathbb{E}[\Vert \nabla_xf(x_t, y_t) \Vert^2] + \frac{L}{2}\tau_1^2\sigma^2 \\ \nonumber
     \leq& \mathbb{E}[g(x_t)-g^*] - \tau_1\mathbb{E}[\langle \nabla g(x_t), \nabla_xf(x_t,y_t) \rangle] +\frac{\tau_1}{2}\mathbb{E}[\Vert \nabla_xf(x_t, y_t) \Vert^2] + \frac{L}{2}\tau_1^2\sigma^2\\ \label{x contraction}
     \leq& \mathbb{E}[g(x_t)-g^*] - \frac{\tau_1}{2}\mathbb{E}\Vert \nabla g(x_t) \Vert^2 + \frac{\tau_1}{2}\mathbb{E}\Vert \nabla_x f(x_t, y_t) - \nabla g(x_t)\Vert^2 +\frac{L}{2}\tau_1^2\sigma^2,
\end{align}
where in the second inequality we use Assumption \ref{stochastic gradients}, and in the third inequality we use $\tau_1\leq 1/L$. Because $-f(x_{t+1}, y)$ is $l$-smooth and $\mu_1$-PL, by Lemma \ref{min linear}, when $\tau_1\leq 1/l$ we have 
\begin{align} \nonumber
    \mathbb{E}[g(x_{t+1}) - f(x_{t+1}, y_{t+1})] &\leq (1-\mu_2\tau_2)\mathbb{E}[g(x_{t+1}) - f(x_{t+1}, y_t)] + \frac{l}{2}\tau_2^2\sigma^2\\  \label{y conctraction}
    &\leq (1-\mu_2\tau_2)\mathbb{E}[g(x_t) - f(x_t, y_t) +f(x_t, y_t) - f(x_{t+1}, y_t) +g(x_{t+1}) -g(x_t)] + \frac{l}{2}\tau_2^2\sigma^2
\end{align}
Because of lipschitz continuity of the gradient, we can bound $f(x_t,y_t) - f(x_{t+1},y_t)$ as
\begin{align*}
    f(x_t,y_t) - f(x_{t+1},y_t)& \leq -\langle \nabla_x f(x_t,y_t), x_{t+1}-x_t\rangle + \frac{l}{2}\Vert x_{t+1} - x_t \Vert^2 \\
    &\leq   \tau_1\langle \nabla_x f(x_t,y_t), G_x(x_t, y_t,\xi_{t1}) \rangle +\frac{l}{2}\tau_1^2\Vert G_x(x_t, y_t,\xi_{t1}) \Vert^2.
\end{align*}
Taking expectation of both side and use Assumption \ref{stochastic gradients},
\begin{equation} \label{x step bound}
    \mathbb{E}[f(x_t, y_t) - f(x_{t+1}, y_t)] \leq (\tau_1 +\frac{l}{2}\tau_1^2)\mathbb{E}\Vert \nabla_x f(x_t, y_t)\Vert^2 + \frac{l}{2}\tau_1^2\sigma^2.
\end{equation}
Also from (\ref{x contraction}) , 
\begin{equation} \label{x step bound 2}
    \mathbb{E}[g(x_{t+1}) - g(x_t)] \leq - \frac{\tau_1}{2}\mathbb{E}\Vert \nabla g(x_t) \Vert^2 + \frac{\tau_1}{2}\mathbb{E}\Vert \nabla_x f(x_t, y_t) - \nabla g(x_t)\Vert^2 +\frac{L}{2}\tau_1^2\sigma^2.
\end{equation}
Combining (\ref{y conctraction}), (\ref{x step bound}) and (\ref{x step bound 2}),
\begin{align} \nonumber
    \mathbb{E}[g(x_{t+1}) - f(x_{t+1}, y_{t+1})] \leq& (1-\mu_2\tau_2)\mathbb{E}[g(x_{t}) - f(x_{t}, y_t)] + (1-\mu_2\tau_2)(\tau_1+\frac{l}{2}\tau_1^2)\mathbb{E}\Vert \nabla_x f(x_t, y_t)\Vert^2 - \\ \nonumber
   & (1-\mu_2\tau_2)\frac{\tau_1}{2}\mathbb{E}\Vert \nabla g(x_t)\Vert^2 +(1-\mu_2\tau_2)\frac{\tau_1}{2}\mathbb{E}\Vert \nabla_x f(x_t, y_t) - \nabla g(x_t)\Vert^2 + \\ \label{y contraction 2}
   &(1-\mu_2\tau_2)\frac{L+l}{2}\tau_1^2\sigma^2 +\frac{l}{2}\tau_2^2\sigma^2.
\end{align}
Combining (\ref{x contraction}) and (\ref{y contraction 2}), we have for $\forall \lambda>0$
\begin{align} \nonumber
    a_{t+1} + \lambda b_{t+1} \leq& a_t - \left[\frac{\tau_1}{2} + \lambda (1-\mu_2\tau_1)\frac{\tau_1}{2} \right] \mathbb{E}\Vert \nabla g(x_t)\Vert^2 + \lambda (1-\mu_2\tau_2)b_t +\\ \nonumber
    & \left[\frac{\tau_1}{2} + \lambda(1-\mu_2\tau_2)\frac{\tau_1}{2} \right]\mathbb{E}\Vert \nabla_x f(x_t, y_t) - \nabla g(x_t)\Vert^2  
     + \lambda(1-\mu_2\tau_2)\left(\tau_1 +\frac{l}{2}\tau_1^2  \right)\mathbb{E}\Vert \nabla_x f(x_t, y_t)\Vert^2 +\\ \nonumber
     & \lambda(1-\mu_2\tau_2)\frac{L+l}{2}\tau_1^2\sigma^2 + \frac{l}{2}\lambda\tau_2^2\sigma^2 + \frac{L}{2}\tau_1^2\sigma^2 \\ \nonumber
     \leq& a_t - \left[\frac{\tau_1}{2} + \lambda (1-\mu_2\tau_1)\frac{\tau_1}{2} -\lambda(1+\beta)(1-\mu_2\tau_2)\left(\tau_1 + \frac{l}{2}\tau_1^2\right) \right] \mathbb{E} \Vert \nabla g(x_t)\Vert^2 +\\ \nonumber
     & \lambda (1-\mu_2\tau_2)b_t + \left[\frac{\tau_1}{2} + \lambda(1-\mu_2\tau_2)\frac{\tau_1}{2} + \lambda \left(1+\frac{1}{\beta}\right)(1-\mu_2\tau_2)\left(\tau_1 +\frac{l}{2}\tau_1^2\right) \right]\mathbb{E}\Vert \nabla_x f(x_t, y_t) - \nabla g(x_t)\Vert^2+ \\ \label{a+b}
     & \lambda(1-\mu_2\tau_2)\frac{L+l}{2}\tau_1^2\sigma^2 + \frac{l}{2}\lambda\tau_2^2\sigma^2 + \frac{L}{2}\tau_1^2\sigma^2,
\end{align}
where in the second inequality we use Young's Inequality and $\beta>0$. Now it suffices to bound $\Vert g(x_t)\Vert^2$ and $\Vert \nabla_xf(x_t, y_t)-\nabla g(x_t)\Vert^2$ by $a_t$ and $b_t$.  With Lemma \ref{g smooth}, we have:
\begin{equation}
    \Vert \nabla_x f(x_t, y_t)-\nabla g(x_t)\Vert^2 = \Vert \nabla_x f(x_t, y_t)-\nabla_x f(x_t, y^*(x_t))\Vert^2 \leq l^2\Vert y^*(x_t) - y_t \Vert^2, \label{A6.2}
\end{equation}
for any $y^*(x_t) \in \arg\max_yf(x_t,y)$. Now we fix $y^*(x_t)$ to be the projection of $y_t$ on the the set $\arg\max_yf(x_t,y)$. Because $-f(\bm{x_t},\cdot)$ satisfies PL condition with $\mu_2$, and Lemma \ref{PL to EB QG} therefore indicates it also satisfies quadratic growth condition with $\mu_2$, i.e.
\begin{equation}
    \Vert y^*(x_t) - y_t \Vert^2 \leq \frac{2}{\mu_2}[g(x_t)-f(x_t, y_t)],
\end{equation}
along with (\ref{A6.2}), we get 
\begin{equation}
    \Vert \nabla_x f(x_t, y_t)-\nabla g(x_t)\Vert^2 \leq \frac{2l^2}{\mu_2}[g(x_t)-f(x_t, y_t)]. \label{gradient different bound}
\end{equation}
Because $g$ satisfies PL condition with $\mu_1$ by Lemma \ref{g PL}, 
\begin{equation} \label{gradient g bound}
    \Vert \nabla g(x_t) \Vert^2 \geq 2\mu_1[g(x_t) - g^*].
\end{equation}
Plug (\ref{gradient different bound}) and (\ref{gradient g bound}) into (\ref{a+b}), we can get
\begin{align}   \nonumber
    a_{t+1} + \lambda b_{t+1} \leq& \Big\{1-\mu_1\big[\tau_1 + \lambda(1-\mu_2\tau_2)\tau_1 - \lambda(1+\beta)(1-\mu_2\tau_2)(2\tau_1 + l\tau_1^2)\big]  \Big\}a_t + \\ \nonumber
    &\lambda \Big\{1-\mu_2\tau_2 + \frac{l^2\tau_1}{\mu_2\lambda} + (1-\mu_2\tau_2)\frac{l^2}{\mu_2}\tau_1 + (1+\frac{1}{\beta})(1-\mu_2\tau_2)\frac{l^2}{\mu_2}(2\tau_1 + l\tau_1^2) \Big\}b_t + \\ \label{contraction}
    & \lambda(1-\mu_2\tau_2)\frac{L+l}{2}\tau_1^2\sigma^2 + \frac{l}{2}\lambda\tau_2^2\sigma^2 + \frac{L}{2}\tau_1^2\sigma^2.
\end{align}
\end{proof}

\textbf{Proof of Theorem \ref{main}}
\begin{proof} 
In the setting of Theorem 1, $\tau_1^t = \tau_1$ and $\tau_2^t = \tau_2, \forall t$. By Thoerem \ref{contraction theorem}, We only need to choose $\tau_1$, $\tau_2$, $\lambda$ and $\beta$ to let $k_1, k_2 <1$. Here we first choose $\beta = 1$ and $\lambda = 1/10$.  Then
\begin{align} \nonumber
    k_1 =& 1-\mu_1\big[\tau_1 + \lambda(1-\mu_2\tau_2)\tau_1 - \lambda(1+\beta)(1-\mu_2\tau_2)(2\tau_1 + l\tau_1^2)\big] \\
    \leq & 1-\mu_1\big\{\tau_1 - \lambda(1-\mu_2\tau_2)\tau_1[(1+\beta)(2+l\tau_1)-1] \big\} \leq 1 - \frac{1}{2}\tau_1\mu_1,
\end{align}
where in the last inequality we just plug in $\beta$ and $\lambda$ and use $l\tau_1 \leq 1$. Also,
\begin{align} \nonumber
    k_2 =& 1-\mu_2\tau_2 + \frac{l^2\tau_1}{\mu_2\lambda} + (1-\mu_2\tau_2)\frac{l^2}{\mu_2}\tau_1 + (1+\frac{1}{\beta})(1-\mu_2\tau_2)\frac{l^2}{\mu_2}(2\tau_1 + l\tau_1^2) \\ \nonumber
    \leq & 1 - \frac{l^2\tau_1}{\mu_2}\left\{\frac{\mu_2^2\tau_2}{\tau_1l^2} -\frac{1}{\lambda} - (1-\mu_2\tau_2)\left[1+\left(1+\frac{1}{\beta}\right)(2+l\tau_1)  \right]   \right\} \\
    \leq & 1 - \frac{l^2\tau_1}{\mu_2},
\end{align}
where in the last inequality we plug in $\beta$ and $\lambda$ and we use $\frac{\mu_2^2\tau_2}{\tau_1l^2} \leq 18$ by our choice of $\tau_1$. Note that $\frac{1}{2}\tau_1\mu_1 < \frac{l^2\tau_1}{\mu_2}$, because $\left(\frac{1}{2}\tau_1\mu_1\right) /\left( \frac{l^2\tau_1}{\mu_2}\right) = \frac{\mu_1\mu_2}{2l^2} < 1$. Define $P_t : = a_t + \frac{1}{10}b_t$, and by Theorem \ref{contraction theorem}, 
\begin{equation*}
    P_{t+1} \leq \left(1-\frac{1}{2}\tau_1\mu_1 \right)P_t + \frac{(1-\mu_2\tau_2)(L+l)\tau_1^2}{20}\sigma^2 + \frac{l\tau_2^2}{20}\sigma^2 + \frac{L\tau_1^2}{2}\sigma^2.
\end{equation*}
With some simple computation, 
\begin{equation*}
    P_t \leq (1-\frac{1}{2}\mu_1\tau_1)^tP_0 + \frac{(1-\mu_2\tau_2)(L+l)\tau_1^2+l\tau_2^2+10L\tau_1^2}{10\mu_1\tau_1}\sigma^2.
\end{equation*}
We verify that $\tau_1\leq 1/L$ by noting: $\tau_1 \leq \frac{\mu_2^2\tau_2}{18l^2} \leq \frac{\mu_2^2}{18l^3}\leq \frac{\mu_2}{2l^2} $ and $L = l+\frac{l^2}{\mu_2}\leq \frac{2l^2}{\mu_2}$.
\end{proof}

\textbf{Proof of Theorem \ref{main deterministic}}
\begin{proof}
The first part of Theorem \ref{main deterministic} is a direct corollary of Theorem \ref{main} by setting $\sigma = 0$. We show the second part by noting that 
\begin{equation}
    \left\|x_{t+1}-x_{t}\right\|^{2}=\tau_{1}^{2}\left\|\nabla_{x} f\left(x_{t}, y_{t}\right)\right\|^{2}, \text{and} \left\|y_{t+1}-y_{t}\right\|^{2}=\tau_{2}^{2}\left\|\nabla_{y} f\left(x_{t+1}, y_{t}\right)\right\|^{2}.
\end{equation}
Also,
\begin{align} \nonumber
    \Vert \nabla_y f(x_{t+1}, y_t)\Vert^2 \leq& \Vert\nabla_y f(x_{t}, y_t)\Vert^2 + \Vert\nabla_y f(x_{t+1}, y_t)-\nabla_y f(x_{t}, y_t)\Vert^2 \\ \nonumber
    \leq & \Vert \nabla_y f(x_t, y_t)-\nabla_yf(x_t, y^*(x_t))\Vert^2 + l^2\Vert x_{t+1} - x_t\Vert^2\\ \nonumber
    \leq & l^2\Vert y_t - y^*(x_t)\Vert^2 + l^2\Vert x_{t+1} - x_t\Vert^2\\ \label{point x bound}
    \leq& \frac{2l^2}{\mu_2}b_t + l^2\Vert x_{t+1} - x_t\Vert^2 = \frac{2l^2}{\mu_2}b_t + l^2\tau_1^2\Vert \nabla_x f(x_t, y_t)\Vert^2,
\end{align}
where in the second inequality $y^*(x_t)$ is the projection of $y_t$ on the the set $\arg\max_yf(x_t,y)$ and $\nabla_y f(x_t, y^*(x_t)) = 0$, in the third inequality we use lipschtiz continuity of gradient, and in the last inequality we use quadratic growth condition. Also,
\begin{align} \nonumber
    \Vert \nabla_x f(x_t, y_t)\Vert^2 \leq& \Vert \nabla g(x_t)\Vert^2 + \Vert \nabla_x f(x_t, y_t)-\nabla g(x_t)\Vert^2 \\ \nonumber
    = & \Vert \nabla g(x_t) - \nabla g(x^*) \Vert^2 + \Vert \nabla_x f(x_t, y_t)-\nabla g(x_t)\Vert^2 \\ \nonumber
    \leq & L^2\Vert x_t - x^*\Vert^2 + l^2 \Vert y^*(x_t) - y_t \Vert^2 \\ \label{point y bound}
    \leq& \frac{2L^2}{\mu_1}a_t + \frac{2l^2}{\mu_2}b_t,
\end{align}
where in the first equality $x^*$ is the projection of $x_t$ on the set $\arg\min_xg(x)$ and $\nabla g(x^*) = 0$, in the second inequality $y^*(x_t)$ is the projection of $y_t$ on the the set $\arg\max_yf(x_t,y)$ and $\nabla g(x_t) = \nabla_x f(x_t, y_t)$, and in the last inequality we use quadratic growth condition. Therefore with (\ref{point x bound}) and (\ref{point y bound}),
\begin{align} \nonumber
    \left\|x_{t}-x^{*}\right\|^{2}+\left\|y_{t}-y^{*}\right\|^{2} \leq& \tau_{1}^{2}\left\|\nabla_{x} f\left(x_{t}, y_{t}\right)\right\|^{2} + \tau_{2}^{2}\left\|\nabla_{y} f\left(x_{t+1}, y_{t}\right)\right\|^{2}\\ \nonumber
    \leq& (1 + \tau_2^2l^2)\tau_1^2\Vert \nabla_xf(x_t, y_t)\Vert^2 + \frac{2l^2}{\mu_2}\tau_2^2b_t\\ \nonumber
    \leq& \frac{2(1 + \tau_2^2l^2)\tau_1^2L^2}{\mu_1}a_t+ \frac{2(1 + \tau_2^2l^2)\tau_1^2l^2+2l^2\tau_2^2}{\mu_2}b_t\\ \nonumber
    \leq & \left[\frac{2(1 + \tau_2^2l^2)\tau_1^2L^2}{\mu_1}+ \frac{20(1 + \tau_2^2l^2)\tau_1^2l^2+20l^2\tau_2^2}{\mu_2}\right]P_0c^t,
\end{align}
where $c = 1-\frac{\mu_1\mu_2^2}{36l^3}$. Letting $\alpha_1 = \left[\frac{2(1 + \tau_2^2l^2)\tau_1^2L^2}{\mu_1}+ \frac{20(1 + \tau_2^2l^2)\tau_1^2l^2+20l^2\tau_2^2}{\mu_2}\right]P_0$,  we have 
\begin{equation*}
    \left\|x_{t+1}-x_{t}\right\|+\left\|y_{t+1}-y_{t}\right\| \leq \sqrt{2 \alpha_{1}} c^{t / 2}.
\end{equation*}
For $n\geq t$, 
\begin{equation*}
    \left\|x_{n}-x_{t}\right\|+\left\|y_{n}-y_{t}\right\| \leq \sum_{i=t}^{n-1}\left\|x_{i+1}-x_{i}\right\|+\left\|y_{i+1}-y_{i}\right\| \leq \sqrt{2 \alpha_{1}} \sum_{i=t}^{\infty} c^{i / 2} \leq \frac{\sqrt{2 \alpha_{1}} c^{t / 2}}{1-\sqrt{c}},
\end{equation*}
so $\{(x_t, y_t)\}_t$ converges and by first part of this theorem the limit $(x^*, y^*)$ must be a saddle point. Thus we have
\begin{equation*}
    \left\|x_{t}-x^{*}\right\|^{2}+\left\|y_{t}-y^{*}\right\|^{2} \leq \frac{2 \alpha_{1}}{(1-\sqrt{c})^{2}} c^{t} = \alpha c^{t}P_0,
\end{equation*}
with $\alpha = 2\left[\frac{2(1 + \tau_2^2l^2)\tau_1^2L^2}{\mu_1}+ \frac{20(1 + \tau_2^2l^2)\tau_1^2l^2+20l^2\tau_2^2}{\mu_2}\right]/(1-\sqrt{c})^{2}$.
\end{proof}

\textbf{Proof of Theorem \ref{main diminishing}}
\begin{proof}
 First note that since $\tau_1^t \leq \mu_2^2/18l^2$, $\tau_2^t = \frac{18 l^{2} \beta}{\mu_{2}^{2}(\gamma+t)} = \frac{18l^2\tau_1^t}{\mu_2^2} \leq \frac{1}{l}$. Similar to the proof of Theorem \ref{main}, by choosing $\beta = 1$ and $\lambda = 1/10$ in the Theorem \ref{contraction theorem}, we have $\min\{k_1, k_2\} = \frac{1}{2}\mu_1\tau_1^t$. We prove the theorem by induction. When t = 1, it is naturally satisfied by definition of $\nu$. We assume that $P_t \leq \frac{\nu}{\gamma+t}$. Then by Theorem \ref{contraction theorem}, 
\begin{align} \nonumber
    P_{t+1} \leq& \left(1-\frac{1}{2}\mu_1\tau_1\right)P_t + \lambda(1-\mu_2\tau_2^t)\frac{L+l}{2}(\tau_1^t)^2\sigma^2 + \frac{l}{2}\lambda(\tau_2^t)^2\sigma^2 + \frac{L}{2}(\tau_1^t)^2\sigma^2 \\ \nonumber
    \leq& \frac{\gamma+t-\frac{1}{2}\mu_1\beta}{\gamma+t}\frac{\nu}{\gamma+t}+ \left[\frac{(L+l)\beta^2}{20(\gamma+t)^2} + \frac{18^2l^5\beta^2}{20\mu_2^4(\gamma+t)^2}+\frac{L\beta^2}{2(\gamma+t)^2}  \right]\sigma^2 \\ \label{nu ineq}
    \leq& \frac{\gamma+t-1}{(\gamma+t)^2}\nu - \frac{\frac{1}{2}\mu_1\beta - 1}{(\gamma+t)^2}\nu + \left[\frac{(L+l)\beta^2}{20(\gamma+t)^2} + \frac{18^2l^5\beta^2}{20\mu_2^4(\gamma+t)^2}+\frac{L\beta^2}{2(\gamma+t)^2}  \right]\sigma^2 \\ \nonumber
    \leq& \frac{\nu}{\gamma+t+1},
\end{align}
where in the second inequality we plug in $\tau_1^t$ and $\tau_2^t$, in the last inequality we use $(\gamma+t+1)(\gamma+t-1)\leq (\gamma+t)^2$ and the fact that sum of last two terms in (\ref{nu ineq}) is no greater than 0 by our choice of $\nu$. 
\end{proof}

%% file: appendix3.tex
\section{Proofs for Section \ref{Sec4}}
\label{appendix3}

\textbf{Proof of Theorem \ref{svrg thm 2}}
\begin{proof} 
Because the proof is long, we break the proof into three parts for the convenience of understanding the intuition behind it. 

\textbf{Part 1}.

 Consider in one outer loop $k$. Define $a_{t,j} = \mathbb{E}[g(x_{t,j})-g^*]$,  $b_{t,j} = \mathbb{E}[g(x_{t,j}) - f(x_{t,j}, y_{t,j})]$, $\Tilde{a}_t = \mathbb{E}[g(\Tilde{x}_t)-g^*]$ and  $\Tilde{b}_t = \mathbb{E}[g(\Tilde{x}_t) - f(\Tilde{x}_t, \Tilde{y}_t)]$. We omit the subscript $t$ for now.  We denote the stochastic gradients as
 \begin{align*}
     G_x(x_j, y_j) &= \nabla_x f_{i_j}(x_{j},y_{j}) - \nabla_x f_{i_j}(\Tilde{x}, \Tilde{y}) +\nabla_x f(\Tilde{x}, \Tilde{y}),\\
     G_y(x_j, y_{j+1})& = \nabla_y f_{i_j}(x_{j+1}, y_{j}) - \nabla_y f_{i_j}(\Tilde{x}, \Tilde{y}) +\nabla_y f(\Tilde{x}, \Tilde{y}).
 \end{align*}
 
 Note that these are unbiased stochastic gradients. Similar to the proof of Theorem \ref{contraction theorem} (replace $\sigma^2$ in (\ref{x contraction}) ), with $\tau_1 \leq 1/L$, we have 
 \begin{equation} \label{vr a contraction}
     a_{j+1} \leq a_j - \frac{\tau_1}{2}\mathbb{E}\Vert \nabla g(x_j) \Vert^2 + \frac{\tau_1}{2}\mathbb{E}\Vert \nabla_x f(x_j, y_j) - \nabla g(x_j)\Vert^2 +\frac{L}{2}\tau_1^2\mathbb{E}\|G_x(x_j, y_j)- \nabla_xf(x_j, y_j)\|^2
 \end{equation}
By Lemma \ref{min linear}, with $\tau_2 \leq 1/l$,
\begin{equation} \label{vr b contraction}
    b_{j+1} \leq \mathbb{E}[g(x_{j+1}) - f(x_{j+1}, y_j)] - \frac{\tau_2}{2}\mathbb{E}\|\nabla_y f(x_{j+1}, y_j)\|^2 +\frac{l}{2}\tau_2^2\mathbb{E}\|G_y(x_{j+1}, y_j) - \nabla_y f(x_{j+1}, y_j)\|^2
\end{equation}

 Furthermore, we bound the distance to the $\Tilde{x} = x_0$ as 
 \begin{align} \nonumber
     \mathbb{E}\| x_{j+1} - \Tilde{x}\|^2 &= \mathbb{E}\|x_j - \tau_1G_x(x_j, y_j)-\Tilde{x}  \|^2 \\  \nonumber
     & = \mathbb{E}\|x_j-\Tilde{x}\|^2 + 2\mathbb{E}\langle x_j - \Tilde{x}, \tau_1\nabla_xf(x_j, y_j)\rangle + \tau_1^2\mathbb{E}\|\nabla_x f(x_j, y_j) \|^2 + \tau_1^2\mathbb{E}\|G_x(x_j, y_j) - \nabla_x f(x_j, y_j) \|^2 \\ \label{dist contraction 1}
     & \leq (1+\tau_1\beta_1)\mathbb{E}\|x_j -\Tilde{x}\|^2 + \left(\tau_1^2+\frac{\tau_1}{\beta_1}\right)\mathbb{E}\|\nabla_xf(x_j, y_j)\|^2 + \tau_1^2 \mathbb{E}\|G_x(x_j, y_j)-\nabla_xf(x_j, y_j)\|^2,
 \end{align}
 where in the last inequality we use Young's inequality to the inner product and $\beta_1 >0$ is a constant which we will determine later.  Similarly, 
 \begin{equation} \label{dist contraction 2}
     \mathbb{E}\|y_{j+1} - \Tilde{y}\|^2 \leq (1+\tau_2\beta_2)\mathbb{E}\|y_j -\Tilde{y}\|^2 + \left(\tau_2^2+\frac{\tau_2}{\beta_2}\right)\mathbb{E}\| \nabla_yf(x_{j+1}, y_j)\|^2 + \tau_2^2\mathbb{E}\| G_y(x_{j+1}, y_j) - \nabla_y f(x_{j+1}, y_j)\|^2
 \end{equation}
 where in the last inequality we use Young's inequality to the inner product and $\beta_2 >0$ is a constant. We are going to construct a potential function 
 \begin{equation}
     R_j = a_j +\lambda b_j + c_j\|x_j - \Tilde{x}\|^2 +  d_j \|y_j - \Tilde{y}\|^2,
 \end{equation}
 and we will determine $\lambda, c_j$ and $d_j$ later. Combine (\ref{vr a contraction}), (\ref{vr b contraction}) and (\ref{dist contraction 2}), 
 \begin{align} \nonumber
     R_{j+1} \leq & a_j - \frac{\tau_1}{2}\mathbb{E}\Vert \nabla g(x_j) \Vert^2 + \frac{\tau_1}{2}\mathbb{E}\Vert \nabla_x f(x_j, y_j) - \nabla g(x_j)\Vert^2 +\frac{L}{2}\tau_1^2\mathbb{E}\|G_x(x_j, y_j)- \nabla_xf(x_j, y_j)\|^2 + \\ \nonumber
     & \lambda\mathbb{E}[g(x_{j+1}) - f(x_{j+1}, y_j)] - \frac{\lambda\tau_2}{2}\mathbb{E}\|\nabla_y f(x_{j+1}, y_j)\|^2 + \\ \nonumber
     & c _{j+1}\mathbb{E}\|x_{j+1} - \Tilde{x}\|^2 + \left(d_{j+1}+\frac{\lambda l}{2}   \right)\tau_2^2\mathbb{E}\| G_y(x_{j+1}, y_j) - \nabla_y f(x_{j+1}, y_j)\|^2 +\\ \label{vr R1}
     & d_{j+1}(1+\tau_2\beta_2)\mathbb{E}\|y_j -\Tilde{y}\|^2 + d_{j+1}\left(\tau_2^2+\frac{\tau_2}{\beta_2}\right)\mathbb{E}\| \nabla_yf(x_{j+1}, y_j)\|^2 
 \end{align}
 Then we bound the variance of the stochastic gradients,
 \begin{align} \nonumber
   \mathbb{E} \Vert G_y(x_{j+1}, y_j) - \nabla_y f(x_{j+1}, y_j)\Vert^2 &= \mathbb{E}\Vert\nabla_y f_{i_j}(x_{j+1}, y_{j}) - \nabla_y f_{i_j}(\Tilde{x}, \Tilde{y}) +\nabla_y f(\Tilde{x}, \Tilde{y}) - \nabla_y f(x_{j+1}, y_j)\Vert^2\\  \label{vr variance bdn1}
    &\leq \mathbb{E}\Vert\nabla_y f_{i_j}(x_{j+1}, y_{j}) - \nabla_y f_{i_j}(\Tilde{x}, \Tilde{y})\Vert^2 \leq l^2\mathbb{E}\| x_{j+1} - \Tilde{x}\|^2 + l^2 \mathbb{E}\|y_j  - \Tilde{y}\|^2
\end{align}
where in the first inequality we use $\mathbb{E}[\nabla_y f_{i_j}(x_{j+1}, y_{j}) - \nabla_y f_{i_j}(\Tilde{x}, \Tilde{y})] = \nabla_y f(x_{j+1}, y_j)- \nabla_y f(\Tilde{x}, \Tilde{y}) $. Similarly, 
\begin{equation} \label{vr variance bdn2}
    \mathbb{E}\| G_x(x_j, y_j) - \nabla_xf(x_j, y_j)\|^2 \leq l^2 \mathbb{E}\|x_j -\Tilde{x}\|^2 + l^2\mathbb{E}\|y_j - \Tilde{y}\|^2. 
\end{equation}
 Plugging (\ref{vr variance bdn1}) into (\ref{vr R1}), 
 \begin{align} \nonumber
     R_{j+1} \leq & a_j - \frac{\tau_1}{2}\mathbb{E}\Vert \nabla g(x_j) \Vert^2 + \frac{\tau_1}{2}\mathbb{E}\Vert \nabla_x f(x_j, y_j) - \nabla g(x_j)\Vert^2 +\frac{L}{2}\tau_1^2\mathbb{E}\|G_x(x_j, y_j)- \nabla_xf(x_j, y_j)\|^2 + \\ \nonumber
     & \lambda\mathbb{E}[g(x_{j+1}) - f(x_{j+1}, y_j)] - \frac{\lambda\tau_2}{2}\mathbb{E}\|\nabla_y f(x_{j+1}, y_j)\|^2 + \\ \nonumber
     & \left[ c _{j+1} + \left(d_{j+1} + \frac{\lambda l}{2} \right)l^2\tau_2^2\right]\mathbb{E}\|x_{j+1} - \Tilde{x}\|^2  +\\ 
     & \left[d_{j+1}(1+\tau_2\beta_2)+\left(d_{j+1} + \frac{\lambda l}{2} \right)l^2\tau_2^2\right]\mathbb{E}\|y_j -\Tilde{y}\|^2 + d_{j+1}\left(\tau_2^2+\frac{\tau_2}{\beta_2}\right)\mathbb{E}\| \nabla_yf(x_{j+1}, y_j)\|^2.
 \end{align}
 Then we plug in (\ref{dist contraction 1}) and rearrange, 
  \begin{align} \nonumber
     R_{j+1} \leq & a_j - \frac{\tau_1}{2}\mathbb{E}\Vert \nabla g(x_j) \Vert^2 + \left[c_{j+1} + \left(d_{j+1} + \frac{\lambda l}{2}  \right)l^2\tau_2^2  \right]\left(\tau_1^2 + \frac{\tau_1}{\beta_1}  \right)\mathbb{E}\|\nabla_x f(x_j, y_j)\|^2 + \frac{\tau_1}{2}\mathbb{E}\|\nabla_xf(x_j, y_j) - \nabla g(x_j)\|^2 + \\ \nonumber
     & \lambda\mathbb{E}[g(x_{j+1}) - f(x_{j+1}, y_j)] -\left[ \frac{\lambda\tau_2}{2} - d_{j+1}\left(\tau_2^2+\frac{\tau_2}{\beta_2}\right)\right]\mathbb{E}\|\nabla_y f(x_{j+1}, y_j)\|^2 + \\ \nonumber
     & \left[ c _{j+1} + \left(d_{j+1} + \frac{\lambda l}{2} \right)l^2\tau_2^2\right](1+\tau_1\beta_1)\mathbb{E}\|x_{j} - \Tilde{x}\|^2 + \left[d_{j+1}(1+\tau_2\beta_2)+\left(d_{j+1} + \frac{\lambda l}{2} \right)l^2\tau_2^2\right]\mathbb{E}\|y_j -\Tilde{y}\|^2  + \\
     & \left[\frac{L}{2}+  c _{j+1} + \left(d_{j+1} + \frac{\lambda l}{2} \right)l^2\tau_2^2  \right]\tau_1^2\mathbb{E}\| G_x(x_j, y_j) - \nabla_xf(x_j, y_j)\|^2
 \end{align}
 Consider the second line. Using PL condition $\|\nabla_yf(x_{j+1}, y_j)\|^2 \geq 2\mu_2 [g(x_{j+1})-f(x_{j+1}, y_j)]$ and assuming $\lambda \geq d_{j+1}(\tau_2+1/\beta_2)$, which we will justify later by our choices of $d_{j+1}$ and $\beta_2$, we have 
 \begin{align*}
     \text{the second line} \leq& \lambda \left[1-\tau_2\mu_2 +\frac{\lambda}{2}d_{j+1}\left(\tau_2^2 + \frac{\tau_2}{\beta_2}  \right)\mu_2  \right]\mathbb{E}[g(x_{j+1}) - f(x_{j+1}, y_j)]\\
     \leq& \lambda \left[1-\tau_2\mu_2 +\frac{\lambda}{2}d_{j+1}\left(\tau_2^2 + \frac{\tau_2}{\beta_2}  \right)\mu_2  \right] \Big\{ b_j + \mathbb{E}\big(f(x_j, y_j) - f(x_{j+1}, y_j)\big) + (a_{j+1}-a_j)\Big\} \\
     \leq& \lambda \left[1-\tau_2\mu_2 +\frac{\lambda}{2}d_{j+1}\left(\tau_2^2 + \frac{\tau_2}{\beta_2}  \right)\mu_2  \right] \Big\{ b_j +\left(\tau_1 + \frac{l}{2}\tau_1^2 \right)\mathbb{E}\|\nabla_x f(x_j, y_j)\|^2 + \\
     &\frac{l}{2}\tau_1^2\mathbb{E}\|G_x(x_j, y_j) - \nabla_xf(x_j, y_j\|^2 - \frac{\tau_1}{2}\mathbb{E}\Vert \nabla g(x_j) \Vert^2 +  \\
     &\frac{\tau_1}{2}\mathbb{E}\Vert \nabla_x f(x_j, y_j) - \nabla g(x_j)\Vert^2 +\frac{L}{2}\tau_1^2\mathbb{E}\|G_x(x_j, y_j)- \nabla_xf(x_j, y_j)\|^2 \Big\}
 \end{align*}
 where in the last inequality we use (\ref{vr a contraction}) and (\ref{x step bound}). Now we plug this into $R_{j+1}$,
  \begin{align} \nonumber
     R_{j+1} \leq & a_j - \frac{\tau_1}{2}(1+\lambda\zeta)\mathbb{E}\Vert \nabla g(x_j) \Vert^2 + \bigg\{\left[c_{j+1} + \left(d_{j+1} + \frac{\lambda l}{2}  \right)l^2\tau_2^2  \right]\left(\tau_1^2 + \frac{\tau_1}{\beta_1}  \right) + \lambda \zeta \left(\tau_1+\frac{l}{2}\tau_1^2 \right)\bigg\}\mathbb{E}\|\nabla_x f(x_j, y_j)\|^2 +  \\ \nonumber
     & \frac{\tau_1}{2}(1+\lambda\zeta)\mathbb{E}\|\nabla_xf(x_j, y_j) - \nabla g(x_j)\|^2 + \lambda\zeta b_j + \\ \nonumber
     & \left[ c _{j+1} + \left(d_{j+1} + \frac{\lambda l}{2} \right)l^2\tau_2^2\right](1+\tau_1\beta_1)\mathbb{E}\|x_{j} - \Tilde{x}\|^2 + \left[d_{j+1}(1+\tau_2\beta_2)+\left(d_{j+1} + \frac{\lambda l}{2} \right)l^2\tau_2^2\right]\mathbb{E}\|y_j -\Tilde{y}\|^2  + \\
     & \left[\frac{L}{2}+  c _{j+1} + \left(d_{j+1} + \frac{\lambda l}{2} \right)l^2\tau_2^2  +\lambda\zeta\frac{L+l}{2}\right]\tau_1^2\mathbb{E}\| G_x(x_j, y_j) - \nabla_xf(x_j, y_j)\|^2,
 \end{align}
 where we define $\zeta = 1-\tau_2\mu_2 +\frac{\lambda}{2}d_{j+1}\left(\tau_2^2 + \frac{\tau_2}{\beta_2}  \right)\mu_2 $ and $\psi = 1-\zeta$. With $\|\nabla_xf(x_j, y_j)\|^2 \leq 2\|\nabla g(x_j)\|^2 + 2\|\nabla g(x_j) - \nabla_xf(x_j, y_j)\|^2$,
 
   \begin{align} \nonumber
     R_{j+1} \leq & a_j - \bigg\{\frac{\tau_1}{2}(1+\lambda\zeta) - 2\left[c_{j+1} + \left(d_{j+1} + \frac{\lambda l}{2}  \right)l^2\tau_2^2  \right]\left(\tau_1^2 + \frac{\tau_1}{\beta_1}  \right) - 2\lambda \zeta \left(\tau_1+\frac{l}{2}\tau_1^2 \right) \bigg\}\mathbb{E}\Vert \nabla g(x_j) \Vert^2  +  \\ \nonumber
     & \lambda\zeta b_j + \bigg\{\frac{\tau_1}{2}(1+\lambda\zeta) + 2\left[c_{j+1} + \left(d_{j+1} + \frac{\lambda l}{2}  \right)l^2\tau_2^2  \right]\left(\tau_1^2 + \frac{\tau_1}{\beta_1}  \right) - 2\lambda \zeta \left(\tau_1+\frac{l}{2}\tau_1^2 \right) \bigg\}\mathbb{E}\|\nabla_xf(x_j, y_j) - \nabla g(x_j)\|^2 + \\ \nonumber
     & \left[ c _{j+1} + \left(d_{j+1} + \frac{\lambda l}{2} \right)l^2\tau_2^2\right](1+\tau_1\beta_1)\mathbb{E}\|x_{j} - \Tilde{x}\|^2 + \left[d_{j+1}(1+\tau_2\beta_2)+\left(d_{j+1} + \frac{\lambda l}{2} \right)l^2\tau_2^2\right]\mathbb{E}\|y_j -\Tilde{y}\|^2  + \\
     & \left[\frac{L}{2}+  c _{j+1} + \left(d_{j+1} + \frac{\lambda l}{2} \right)l^2\tau_2^2  +\lambda\zeta\frac{L+l}{2}\right]\tau_1^2\mathbb{E}\| G_x(x_j, y_j) - \nabla_xf(x_j, y_j)\|^2.
 \end{align}
Then plugging in (\ref{gradient different bound}), (\ref{gradient g bound}) and (\ref{vr variance bdn2}), we get
   \begin{align} \nonumber
     R_{j+1} \leq & a_j - \bigg\{\tau_1(1+\lambda\zeta) - 4\left[c_{j+1} + \left(d_{j+1} + \frac{\lambda l}{2}  \right)l^2\tau_2^2  \right]\left(\tau_1^2 + \frac{\tau_1}{\beta_1}  \right) - 4\lambda \zeta \left(\tau_1+\frac{l}{2}\tau_1^2 \right) \bigg\}\mu_1 a_j  +  \\ \nonumber
     & \lambda b_j - \lambda\frac{1}{\lambda}\bigg\{\lambda\psi - \frac{l^2\tau_1}{\mu_2}(1+\lambda\zeta) - \frac{4l^2}{\mu_2}\left[c_{j+1} + \left(d_{j+1} + \frac{\lambda l}{2}  \right)l^2\tau_2^2  \right]\left(\tau_1^2 + \frac{\tau_1}{\beta_1}  \right) - \frac{4l^2}{\mu_2}\lambda \zeta \left(\tau_1+\frac{l}{2}\tau_1^2 \right) \bigg\}b_j + \\ \nonumber
     & \bigg\{\left[ c _{j+1} + \left(d_{j+1} + \frac{\lambda l}{2} \right)l^2\tau_2^2\right](1+\tau_1\beta_1)+\left[\frac{L}{2}+  c _{j+1} + \left(d_{j+1} + \frac{\lambda l}{2} \right)l^2\tau_2^2  +\lambda\zeta\frac{L+l}{2}\right]\tau_1^2l^2\bigg\}\mathbb{E}\|x_{j} - \Tilde{x}\|^2 + \\ \label{vr R_t}
     & \bigg\{\left[d_{j+1}(1+\tau_2\beta_2)+\left(d_{j+1} + \frac{\lambda l}{2} \right)l^2\tau_2^2\right] +\left[\frac{L}{2}+  c _{j+1} + \left(d_{j+1} + \frac{\lambda l}{2} \right)l^2\tau_2^2  +\lambda\zeta\frac{L+l}{2}\right]\tau_1^2l^2 \bigg\}\mathbb{E}\|y_j -\Tilde{y}\|^2.
 \end{align}
Now we are ready to define sequences $\{c_j\}_j$ and $\{d_j\}_j$. Let $c_N=d_N= 0$, and 
\begin{align*}
    c_j &= \left[ c _{j+1} + \left(d_{j+1} + \frac{\lambda l}{2} \right)l^2\tau_2^2\right](1+\tau_1\beta_1)+\left[\frac{L}{2}+  c _{j+1} + \left(d_{j+1} + \frac{\lambda l}{2} \right)l^2\tau_2^2  +\lambda\zeta\frac{L+l}{2}\right]\tau_1^2l^2, \\
    d_j &= \left[d_{j+1}(1+\tau_2\beta_2)+\left(d_{j+1} + \frac{\lambda l}{2} \right)l^2\tau_2^2\right] +\left[\frac{L}{2}+  c _{j+1} + \left(d_{j+1} + \frac{\lambda l}{2} \right)l^2\tau_2^2  +\lambda\zeta\frac{L+l}{2}\right]\tau_1^2l^2.
\end{align*}
We further define
\begin{align} \label{vr m1}
    m_j^1 :=& \tau_1(1+\lambda\zeta) - 4\left[c_{j+1} + \left(d_{j+1} + \frac{\lambda l}{2}  \right)l^2\tau_2^2  \right]\left(\tau_1^2 + \frac{\tau_1}{\beta_1}  \right) - 4\lambda \zeta \left(\tau_1+\frac{l}{2}\tau_1^2 \right), \\ \label{vr m2}
    m_j^2 :=& \frac{1}{\lambda}\bigg\{\lambda\psi - \frac{l^2\tau_1}{\mu_2}(1+\lambda\zeta) - \frac{4l^2}{\mu_2}\left[c_{j+1} + \left(d_{j+1} + \frac{\lambda l}{2}  \right)l^2\tau_2^2  \right]\left(\tau_1^2 + \frac{\tau_1}{\beta_1}  \right) - \frac{4l^2}{\mu_2}\lambda \zeta \left(\tau_1+\frac{l}{2}\tau_1^2 \right)\bigg\}.
\end{align}
Then we can write (\ref{vr R_t}) as 
\begin{equation}
    R_{j+1} \leq  R_j - m_j^1a_j - \lambda m_j^2b_j
\end{equation}
Now we bring back the subscript $t$. Summing the equation from $0$ to $N-1$, 
\begin{equation}
    \sum_{j=0}^{N-1} a_{t,j} + \lambda b_{t,j} \leq\frac{R_0 - R_N}{N \gamma} = \frac{a_{t,0} + \lambda b_{t,0} - a_{t,N} - \lambda b_{t,N}}{N\gamma} =  \frac{\Tilde{a}_t+\lambda\Tilde{b}_t - \Tilde{a}_{t+1} - \lambda\Tilde{b}_{t+1}}{N\gamma},
\end{equation}
where $\gamma := \min_j \{m_j^1, m_j^2\}$, and the first equality is due to $c_N = d_N=0$ and $(x_{t,0}, y_{t,0}) = (\Tilde{x}_t, \Tilde{y}_t)$. Summing $t$ from 0 to $T-1$, we get
\begin{equation}
    \frac{1}{NT}\sum_{t=0}^{T-1}\sum_{j=0}^{N-1}a_{t, j} +\lambda b_{t,j} \leq \frac{\Tilde{a}_0+\lambda \Tilde{b}_0}{NT\gamma} = \frac{a^k + \lambda b^k}{NT\gamma}.
\end{equation}
The left hand side is exactly $a^{k+1}+\lambda b^{k+1}$, because $(x_k, y_k)$ is sampled uniformly from $\{\{(x_{t,j}, y_{t,j})\}_{j=0}^{N-1}\}_{t = 0}^{T-1}$.

\textbf{Part 2}.

It suffices to choose proper $\tau_1$, $\tau_2$, $N$ and $T$ such that $NT\gamma>1$. Driven by the proof, we choose
\begin{equation*}
    \tau_1 = \frac{k_1}{\kappa^2 l}, \quad \beta_1 = k_2\kappa^2l, \quad \tau_2 = \frac{k_3}{l}, \quad \beta_2 = lk_4.
\end{equation*}
We will choose $k_1, k_2, k_3$ and $k_4$ later and we let $k_1, k_2, k_3, k_4 \leq 1$. Plug back to $c_j$ and $d_j$, we have
\begin{align} \nonumber
    c_j =& \left(1+k_1k_2+\frac{k_1^2}{\kappa^4}\right)c_{j+1} + \left[k_3^2(1+k_1k_2) + \frac{k_1^2k_3^2}{\kappa^4} +  (L+l)\frac{k_1^2}{\kappa^4}\left(\frac{k_3^2}{l^2} + \frac{k_3}{l^2k_4} \right)\mu_2  \right]d_{j+1} + \\ \nonumber
    &  \frac{\lambda}{2}lk_3^2(1+k_1k_2) + \frac{L}{2\kappa^4} k_1^2 + \frac{\lambda}{2\kappa^4}lk_1^2k_3^2 + \frac{\lambda}{2\kappa^4}(L+l)k_1^2(1-k_3k_4)\\ \label{vr c bound}
    \leq & \left(1+k_1k_2+\frac{k_1^2}{\kappa^4}\right)c_{j+1} + \left(3k_3^2  + 3\frac{1}{\kappa^3}k_1^2 \right)d_{j+1} +  2\lambda lk_3^2 + (1+2\lambda)\frac{l}{\kappa^3}k_1^2,
\end{align}
where in the last inequality we assume $k_3^2 + \frac{k_3}{k_4}\leq 1$.
\begin{align}  \nonumber
    d_j =& \frac{k_1^2}{\kappa^4}c_{j+1} + \left[ 1+k_3k_4 + k_3^2 + (L+l)\frac{k_1^2}{\kappa^4}\left( \frac{k_3^2}{l^2}+ \frac{k_3}{l^2k_4} \right)\mu_2 + \frac{1}{\kappa^4}k_1^2k_3^2\right]d_{j+1} + \\ \nonumber
    & \frac{\lambda}{2}lk_3^2 + \frac{L}{2\kappa^4}k_1^2 + \frac{\lambda}{2\kappa^4}lk_1^2k_3^2 + \frac{\lambda}{2\kappa^4}(L+l)k_1^2(1-k_3k_4)\\ \label{vr d bound}
    \leq & \frac{k_1^2}{\kappa^4}c_{j+1} + \left( 1+k_3k_4 + 2k_3^2 + \frac{3}{\kappa^3}k_1^2  \right)d_{j+1} + \lambda lk_3^2 + (1+2\lambda)\frac{l}{\kappa^3}k_1^2. 
\end{align}
We define $e_j = \max\{c_j, d_j\}$. Then combining (\ref{vr c bound}) and (\ref{vr d bound}), we easily get 
\begin{equation*}
    e_j \leq \left(1+k_1k_2+k_3k_4 + 3k_3^2 + \frac{4}{\kappa^3}k_1^2\right)e_{j+1} + 2\lambda lk_3^2 + (1+2\lambda)\frac{l}{\kappa^3}k_1^2.
\end{equation*}
As $e_N = 0$, we have
\begin{equation} \label{vr e bound}
    e_0 \leq \left[2\lambda lk_3^2 + (1+2\lambda)\frac{l}{\kappa^3}k_1^2\right] \frac{\left(1+k_1k_2+k_3k_4 + 3k_3^2 + \frac{4}{\kappa^3}k_1^2  \right)^N-1}{k_1k_2+k_3k_4 + 3k_3^2 + \frac{4}{\kappa^3}k_1^2},
\end{equation}
and note that $e_j > e_{j+1}$ so $e_j \leq e_0, \forall j$. Then we want to lower bound $\gamma$. Rearrange (\ref{vr m1}), 
\begin{align} \nonumber
    m_j^1 =& \mu_1\bigg\{\tau_1(1+\lambda -\lambda\tau_2\mu_2) - 2\lambda l^3\tau_2^2\left(\tau_1^2 +\frac{\tau_1}{\beta_1}\right) - 4\lambda\left(\tau_1 + \frac{l}{2}\tau_1^2\right)(1-\tau_2\mu_2) - \\ \nonumber
    &\left[-2\tau_1\left(\tau_2^2+\frac{\tau_2}{\beta_2}\right)\mu_2 + 4\left(\tau_1^2 + \frac{\tau_1}{\beta_1}\right)l^2\tau_2^2 + 8\left(\tau_1 + \frac{l}{2}\tau_1^2\right)\left(\tau_2^2 + \frac{\tau_2}{\beta_2}\right)\mu_2    \right]d_{j+1} - \\ \nonumber
    &4\left(\tau_1^2 + \frac{\tau_1}{\beta_1}\right)c_{j+1}\bigg\}\\
    \geq& \frac{1}{2}\tau_1\mu_1 - \left[\frac{4}{\kappa^4}k_3^2\left(k_1^2 +\frac{k_1}{k_2}\right) + \frac{10\mu_2}{\kappa^2l}k_1\left(k_3^2 + \frac{k_3}{k_4} \right) \right]\frac{\mu_1}{l^2}d_{j+1} - \frac{4}{\kappa^4}\left(k_1^2 + \frac{k_1}{k_2}\right) \frac{\mu_1}{l^2}c_{j+1},
\end{align}
where in the inequality, we use $\lambda = 1/20$ and assume that $\frac{1}{\kappa^2}k_3^2(k_1+\frac{1}{k_2})\leq 10$. Rearranging (\ref{vr m2}), 
\begin{align} \nonumber
    m_j^2 =& \tau_2\mu_2 - \frac{l^2\tau_1}{\mu_2} \left(\frac{1}{\lambda}+1 - \tau_2\mu_2  \right) - \frac{2l^5}{\mu_2}\left(\tau_1^2+\frac{\tau_1}{\beta_1}  \right)\tau_2^2 - \frac{4l^2}{\mu_2}\left( \tau_1+\frac{l}{2}\tau_1^2\right)(1-\tau_2\mu_2)-\\ \nonumber
    & \left[\frac{2}{\lambda}\left(\tau_2^2 + \frac{\tau_2}{\beta_2}\right)\mu_2 + \frac{2}{\lambda}l^2\tau_1\left(\tau_2^2 + \frac{\tau_2}{\beta_2}\right) + \frac{4}{\lambda}\frac{l^4}{\mu_2}\tau_2^2\left(\tau_1^2 + \frac{\tau_1}{\beta_1}\right) + \frac{8l^2}{\lambda\mu_2}\left(\tau_1+\frac{l}{2}\tau_1^2\right)\left(\tau_2^2 + \frac{\tau_2}{\beta_2}\right)\mu_2  \right]d_{j+1} - \\ \nonumber
    &\frac{4}{\lambda}\frac{l^2}{\mu_2}\left(\tau_1^2 + \frac{\tau_1}{\beta_1}\right)c_{j+1} \\
    \geq & \frac{l^2\tau_1}{2\min\{\mu_1, \mu_2\}} - \left[ 200\left(k_3^2 + \frac{k_3}{k_4}\right) + \frac{80}{\kappa^2}\left(k_1^2 + \frac{k_1}{k_2}\right) \right]\frac{\mu_2}{l^2}d_{j+1} - \frac{80}{\kappa^2}\left(k_1^2 + \frac{k_1}{k_2}\right)\frac{\mu_2}{l^2}c_{j+1},
\end{align}
where in the inequality we use $\lambda = 1/20$ and assume $k_1\leq k_3/28$ and $\frac{1}{\kappa^2}k_3^2\left(k_1+ \frac{1}{k_2}\right)\leq 1/4$. Note that $\frac{1}{2}\tau_1\mu_1 = \frac{\mu_1}{2\kappa^2l}k_1$ and $\frac{l^2\tau_1}{2\min\{\mu_1, \mu_2\}} = \frac{l}{2\kappa^2\min\{\mu_1, \mu_2\}}k_1$. Then we have
\begin{align}
    m_j^1 \geq &\frac{1}{\kappa^3}\bigg\{ \frac{1}{2}k_1 - \left[\frac{4}{\kappa^2}k_3^2\left(k_1^2 + \frac{k_1}{k_2} \right) + \frac{10\mu_2}{l}k_1\left(k_3^2 + \frac{k_3}{k_4}\right)\right]\frac{d_{j+1}}{l} - \frac{4}{\kappa^2}\left(k_1^2 + \frac{k_1}{k_2}\right)\frac{c_{j+1}}{l}\bigg\},\\
    m_j^2 \geq & \frac{1}{\kappa} \bigg\{  \frac{1}{2}k_1 - \left[\frac{80}{\kappa^2}\left(k_1^2 + \frac{k_1}{k_2}\right) + 200\left(k_3^2 + \frac{k_3}{k_4}\right)  \right]\frac{d_{j+1}}{l} - \frac{80}{\kappa^2}\left(k_1^2 + \frac{k_1}{k_2}  \right)\frac{c_{j+1}}{l}\bigg\}.
\end{align}
 Letting $k_1/k_2 = k_3/k_4$ and $k_1 = \frac{1}{28}k_3$, we have 
 \begin{equation} \label{vr gamma bound}
     \gamma \geq \frac{1}{\kappa^3}\bigg\{ \frac{1}{56}k_3 - 360\left(k_3^2 + \frac{k_3}{k_4}\right)\frac{e_0}{l}\bigg\},
 \end{equation}
 where we use $c_j, d_j \leq e_0, \forall j$. By plugging in $k_1 = k_3/28$ and $\lambda=1/20$ into (\ref{vr e bound}), we have
 \begin{equation}  \label{e0 bound}
     e_0 \leq l\frac{(1+2k_3k_4+4k_3^2)^N-1}{k_4/k_3+3}.
 \end{equation}
 Plugging this into (\ref{vr gamma bound}), we have 
 \begin{equation}
     \gamma \geq \frac{1}{\kappa^3}\left[ \frac{k_3}{56} - 360\frac{(1+2k_3k_4+4k_3^2)^N-1}{k_4/k_3+3}\left(k_3^2 + k_3/k_4  \right)\right].
 \end{equation}
 We choose $k_4 = k_3^{1/2}$, then 
 \begin{equation}  \label{vr nt bound}
     NT\gamma \geq \frac{1}{\kappa^3}\left[ \frac{k_3}{56} - 360\left((1+2k_3^{3/2}+4k_3^2)^N-1\right)\left(\frac{k_3^2 + k_3^{1/2}}{k_3^{-1/2}+3}\right)\right]NT.
 \end{equation}
 
 \textbf{Part 3}.

 We choose $T=1$, $k_3 = \beta \kappa^{-6}$ and $N = \alpha(2k_3^{3/2}+4k_3^2)^{-1} \geq \frac{\alpha}{2}k_2^{-3/2}$, where $\alpha, \beta$ is irrelevant to $n, l, \mu_1, \mu_2$. Then since $(1+2k_3^{3/2}+4k_3^2)^N \leq e^{\alpha}$, after plugging in  $N$ and $k_3$, we have
 \begin{equation}
     NT\gamma \geq \frac{1}{\kappa^3} \left[\frac{k_3}{56} - 360(e^\alpha -1)(2k_3) \right]\frac{\alpha}{2}k_2^{-3/2} \geq \frac{1}{2}\left[\frac{1}{56} - 2\times 360(e^{\alpha}-1) \right]\alpha\beta^{-1/2} .
 \end{equation}
 Therefore, for choosing $\alpha$ small enough and $\beta$ small enough, we have $NT\gamma \geq 2$. Now it remains to verify several assumptions we made in the proof. The first is $\frac{k_3}{k_4} + k_3^2 \leq 1$. Since $\frac{k_3}{k_4} + k_3^2 = k_3^{1/2} + k_3^2$, this assumption easily holds when $\beta \leq 1/4$. The second assumption we want to verify is $\frac{1}{\kappa^2}k_3^2\left(k_1+ \frac{1}{k_2}\right)\leq 1/4$. Note that 
 \begin{equation*}
     \frac{1}{\kappa^2}k_3^2\left(k_1+ \frac{1}{k_2}\right) =  \frac{1}{\kappa^2}k_3^2\left(k_1+ \frac{k_3}{k_4k_1}\right)=\frac{1}{\kappa^2}k_3^2\left(\frac{1}{28}k_3 + 28k_3^{-1/2} \right).
 \end{equation*}
 So this assumption can also be easily satisfied when $\beta$ is small. The last assumption we need to verify is $\lambda \geq d_{j+1}\left(\tau_2 + \frac{1}{\beta_2}  \right)$. Because $d_{j+1} \leq e_0$ and (\ref{e0 bound}),
 \begin{align*}
     d_{j+1}\left(\tau_2 + \frac{1}{\beta_2}  \right) &\leq l\frac{(1+2k_3k_4+4k_3^2)^N-1}{k_4/k_3+3}\left(\frac{k_3}{l} + \frac{1}{k_4l} \right)\\
     &\leq \left((1+2k_3k_4+4k_3^2)^N-1\right)\left(\frac{k_3^2 + k_3^{1/2}}{k_3^{-1/2}+3}\right)\\
     &\leq 2(e^\alpha-1)k_3.
 \end{align*}
 So this assumption holds when $\alpha$ and $\beta$ are small.

 \end{proof}

\textbf{Proof of Theorem \ref{svrg thm 1}}
\begin{proof} 
We start from Part 3 of the proof of Theorem \ref{svrg thm 2}. We now choose $k_3 = \beta n^{-2/3}$, $N = \alpha(2k_3^{3/2}+4k_3^2)^{-1}$, and $T = \kappa^3n^{-1/3}$ then 
\begin{equation}
    NT\gamma \geq \frac{1}{2}\left[\frac{1}{56} - 2\times360(e^{\alpha}-1) \right]\alpha\beta^{-1/2}
\end{equation}
Therefore, for choosing $\alpha$ small enough and $\beta$ small enough, we have $NT\gamma \geq 2$. Other assumptions can be easily verified by the same way as in the proof of Theorem \ref{svrg thm 2}.
\end{proof}

%% file: appendix4.tex
 \section{AGDA for minimax problems under one-sided PL condition }
 \label{appendix4}

We are here to show that if $-f(x, \cdot)$ satisfies PL condition with constant $\mu$ and $f(\cdot, y)$ may be nonconvex (referred to as PL game by \citet{nouiehed2019solving}), AGDA as presented in Algorithm \ref{agda random} can find $\epsilon$-stationary point of $g(x):= \max_y f(x, y)$ within $\mathcal{O}(\epsilon^{-2})$ iterations. Note that GDmax has complexity $\mathcal{O}(\epsilon^{-2}\log(1/\epsilon))$ on minimax problems under the one-sided PL condition \citep{nouiehed2019solving}; SGDA has complexity $\mathcal{O}(\epsilon^{-2})$ on nonconvex-strongly-concave minimax problems \citep{lin2019gradient}.  Here we define condition number $\kappa = \frac{\mu}{l}$ and $L$ is still defined the same as before. The proof is based on our previous analysis and \citet{lin2019gradient}.

\begin{definition}
$x$ is $\epsilon$-stationary point of a differential function $f$ if $\mathbb{E}\Vert \nabla f(x)\Vert \leq \epsilon$.
\end{definition}

\begin{algorithm}[ht] 
    \caption{AGDA}
    \begin{algorithmic}[1]
        \STATE Input: $(x_0,y_0)$, step sizes $\tau_1>0, \tau_2^t>0$
        \FORALL{$t = 0,1,2,..., T-1$}
            \STATE $x_{t+1}\gets  x_t-\tau_1 \nabla f_x(x_t,y_t)$
            \STATE $y_{t+1}\gets y_t+\tau_2 \nabla f_y(x_{t+1},y_t)$
        \ENDFOR
        \STATE choose $(x^T, y^T)$ uniformly from $\{(x_t, y_t)\}_{t=0}^T$
    \end{algorithmic} \label{agda random}
\end{algorithm}

\begin{theorem}
Suppose Assumption \ref{Lipscthitz gradient} holds and $-f(x, \cdot)$ satisfies PL condition with constant $\mu$ for any $x$. If we run Algorithm \ref{agda random} with $\tau_1 = \frac{1}{20\kappa^2l}$ and $\tau_2 = \frac{1}{l}$, then 
\begin{equation}
    \mathbb{E} \Vert \nabla g(x^{T})\Vert^2 \leq \frac{8}{T+1}[10\kappa^2la_0+\kappa^2lb_0],
\end{equation}
where $a_0 = g(x_0) - g^*$ and $b_0 = g(x_0) - f(x_0, y_0)$.

\end{theorem}

\begin{proof}
For convenience, we still define $b_t = g(x_t) - f(x_t, y_t)$. Since it can be easily verified that $\tau_1 \leq 1/L$, by (\ref{x contraction}) and (\ref{gradient different bound}), we have
\begin{equation} \label{gg}
    g(x_{t+1}) \leq g(x_t) - \frac{\tau_1}{2}\Vert \nabla g(x_t)\Vert^2 + \frac{\tau_1l^2}{\mu_2}b_t.
\end{equation}
By (\ref{y contraction 2}), we have
\begin{align} \nonumber
    b_{t+1} \leq &(1-\mu_2\tau_2)b_t + (1-\mu_2\tau_2)\left(\tau_1+\frac{l}{2}\tau_1^2 \right)\Vert \nabla_x f(x_t, y_t)\Vert^2 - \\ \nonumber
    & (1-\mu_2\tau_2)\frac{\tau_1}{2}\Vert \nabla g(x_t)\Vert^2 + (1-\mu_2\tau_2)\frac{\tau_2}{2}\Vert \nabla_xf(x_t, y_t) - \nabla g(x_t)\Vert^2 \\ \nonumber
    \leq & (1-\mu_2\tau_2)b_t + \left[2(1-\mu_2\tau_2)\left(\tau_1+\frac{l}{2}\tau_1^2 \right)-(1-\mu_2\tau_2)\frac{\tau_2}{2}   \right] \Vert \nabla g(x_t)\Vert^2 + \\ \nonumber
    & \left[2(1-\mu_2\tau_2)\left(\tau_1+\frac{l}{2}\tau_1^2 \right)+(1-\mu_2\tau_2)\frac{\tau_2}{2}  \right]\Vert \nabla_xf(x_t, y_t) - \nabla g(x_t)\Vert^2 \\
    \leq & (1-\mu_2\tau_2)\left[1+\left(5\tau_1+ 2l\tau_1^2  \right)\frac{l^2}{\mu_2}  \right]b_t + (1-\mu_2\tau_2)\left[\frac{3}{2}\tau_1 + l\tau_1^2  \right]\Vert \nabla g(x_t)\Vert^2,
\end{align}
where in the second inequality we use Young's inequality, and in third inequality we use (\ref{gradient different bound}). We write
\begin{equation}
    b_{t+1} = \alpha b_t + \beta \Vert \nabla g(x_k)\Vert^2
\end{equation}
with 
\begin{equation*}
    \alpha  = (1-\mu_2\tau_2)\left[1+\left(5\tau_1+ 2l\tau_1^2  \right)\frac{l^2}{\mu_2}  \right], \qquad
    \beta  = (1-\mu_2\tau_2)\left[\frac{3}{2}\tau_1 + l\tau_1^2  \right].
\end{equation*}
Then 
\begin{equation*}
    b_t \leq \alpha^t b_0 + \beta \sum_{k=0}^{t-1} \alpha^{t-1-k}\Vert \nabla g(x_k)\Vert^2, \quad t \geq 1.
\end{equation*}
Plugging into (\ref{gg}), we have 
\begin{equation}
    g(x_{t+1}) \leq g(x_t) - \frac{\tau_1}{2}\Vert \nabla g(x_t)\Vert^2 + \frac{\tau_1l^2}{\mu_2}\alpha^t b_0 + \frac{\tau_1l^2\beta}{\mu_2} \sum_{k=0}^{t-1} \alpha^{t-1-k}\Vert \nabla g(x_k)\Vert^2, \quad t \geq 1.
\end{equation}
Telescoping and rearranging,
\begin{equation} \label{telescope 2}
    \frac{\tau_1}{2}\sum_{t=0}^{T}\Vert \nabla g(x_t)\Vert^2 - \frac{\tau_1l^2\beta}{\mu_2} \sum_{t=1}^{T}\sum_{k=0}^{t-1} \alpha^{t-1-k}\Vert \nabla g(x_k)\Vert^2 \leq g(x_0) - g(x_{T+1}) + \frac{\tau_1l^2}{\mu_2}b_0\sum_{t=0}^{T}\alpha^t \leq a_0 +   \frac{\tau_1l^2}{\mu_2(1-\alpha)}b_0
\end{equation}
Considering the left hand side of (\ref{telescope 2}),
\begin{equation}
     \sum_{t=1}^{T}\sum_{k=0}^{t-1} \alpha^{t-1-k}\Vert \nabla g(x_k)\Vert^2 = \sum_{k=0}^{T-1}\sum_{t = k+1}^{T}\alpha^{t-1-k}\Vert \nabla g(x_k)\Vert^2\leq \sum_{k=0}^{T-1} \frac{1}{1-\alpha}\Vert \nabla g(x_k)\Vert^2,
\end{equation}
and therefore,
\begin{equation} \label{telescope lhs}
    \frac{\tau_1}{2}\sum_{t=0}^{T}\Vert \nabla g(x_t)\Vert^2 - \frac{\tau_1l^2\beta}{\mu_2} \sum_{t=0}^{T}\sum_{k=0}^{t-1} \alpha^{t-1-k}\Vert \nabla g(x_k)\Vert^2 \geq \sum_{t=0}^{T} \Big\{\frac{1}{2} - \frac{l^2\beta}{\mu_2(1-\alpha)}  \Big\}\tau_1\Vert \nabla g(x_t)\Vert^2.
\end{equation}
We note that
    $\beta  = (1-\mu_2\tau_2)\left[\frac{3}{2}\tau_1 + l\tau_1^2  \right] \leq \frac{5}{2}\tau_1$ because $l/\tau_1\leq 1$ by our choice of $\tau_1$. Also, 
\begin{equation}
    1-\alpha = \mu_2\tau_2 - (1-\mu_2\tau_2)\left(5\tau_1+ 2l\tau_1^2  \right)\frac{l^2}{\mu_2} \geq \mu_2\tau_2 - 7(1-\mu_2\tau_2)\frac{\tau_1l^2}{\mu_2} \geq \frac{1}{2\kappa},
\end{equation}
where in the last inequality we use $\mu_2\tau_2 = 1/\kappa$ and $(1-\mu_2\tau_2)\frac{\tau_1l^2}{\mu_2}= (1-1/\kappa)/(20\kappa) \leq 1/(20\kappa)$. Plugging into (\ref{telescope lhs}),
\begin{equation}
    \frac{\tau_1}{2}\sum_{t=0}^{T}\Vert \nabla g(x_t)\Vert^2 - \frac{\tau_1l^2\beta}{\mu_2} \sum_{t=1}^{T}\sum_{k=0}^{t-1} \alpha^{t-1-k}\Vert \nabla g(x_k)\Vert^2 \geq \frac{\tau_1}{4}\sum_{t=0}^{T}\Vert \nabla g(x_t)\Vert^2.
\end{equation}
Combining with (\ref{telescope 2}), we have
\begin{equation}
    \frac{1}{T+1}\sum_{t=0}^{T}\Vert \nabla g(x_t)\Vert^2 \leq \frac{4}{(T+1)\tau_1} \left[a_0 +   \frac{\tau_1l^2}{\mu_2(1-\alpha)}b_0 \right] \leq \frac{8}{T+1}[10\kappa^2la_0+\kappa^2lb_0],
\end{equation}
where in the inequality we use $1-\alpha \geq 1/(2\kappa)$ again. 

\end{proof}